\documentclass[oneside,reqno]{amsart}
\usepackage[T1]{fontenc}
\usepackage[latin9]{inputenc}
\usepackage{geometry}
\usepackage{mathtools}
\usepackage{enumitem}
\usepackage{amstext}
\usepackage{amsthm}
\usepackage{amssymb}

\makeatletter
\numberwithin{equation}{section}
\numberwithin{figure}{section}
\theoremstyle{plain}
\newtheorem{thm}{\protect\theoremname}
\theoremstyle{plain}
\newtheorem{conjecture}[thm]{\protect\conjecturename}
\theoremstyle{definition}
\newtheorem{problem}[thm]{\protect\problemname}
\theoremstyle{remark}
\newtheorem{rem}[thm]{\protect\remarkname}
\theoremstyle{definition}
\newtheorem{defn}[thm]{\protect\definitionname}
\theoremstyle{plain}
\newtheorem{prop}[thm]{\protect\propositionname}
\theoremstyle{plain}
\newtheorem{lem}[thm]{\protect\lemmaname}

\usepackage{shuffle}

\makeatother

\providecommand{\conjecturename}{Conjecture}
\providecommand{\definitionname}{Definition}
\providecommand{\lemmaname}{Lemma}
\providecommand{\problemname}{Problem}
\providecommand{\propositionname}{Proposition}
\providecommand{\remarkname}{Remark}
\providecommand{\theoremname}{Theorem}

\begin{document}
\address[Minoru Hirose]{Institute for Advanced Research, Nagoya University,  Furo-cho, Chikusa-ku, Nagoya, 464-8602, Japan}
\email{minoru.hirose@math.nagoya-u.ac.jp}
\title{The cyclotomic Grothendieck-Teichm\"{u}ller group and the motivic
Galois group}
\author{Minoru Hirose}
\begin{abstract}
We show that the level 2 case of the cyclotomic Grothendieck-Teichm\"{u}ller
groups introduced by Enriquez coincides with the motivic Galois group
of mixed Tate motives over $\mathbb{Z}[1/2]$.
\end{abstract}

\maketitle
\global\long\def\et{\mathrm{\acute{e}t}}%

\section{Introduction}

The purpose of this paper is to show the coincidence of the cyclotomic
Grothendieck-Teichm\"{u}ller group for level $N=2$ and the motivic
Galois group of mixed Tate motives over $\mathbb{Z}[1/2]$. This work
concerns with the cyclotomic cases of the Grothendieck-Teichm\"{u}ller
theory, which are originated from the level one case. So let us start
from the brief introduction for the original case.

Let $\mathcal{MTM}(\mathbb{Z})$ be the Tannakian category of mixed
Tate motives over $\mathbb{Z}$, and $\mathcal{G}_{1}$ its motivic
Galois group with respect to the canonical fiber functor. Then $\mathcal{G}_{1}$
is decomposed as $\mathcal{G}_{1}=\mathbb{G}_{m}\ltimes\mathcal{U}_{1}$
where $\mathcal{U}_{1}$ is the prounipotent part of $\mathcal{G}_{1}$.
For $p,q\in\{0,1\}$, let $\pi_{1}^{{\rm mot}}(\mathbb{P}^{1}\setminus\{0,1,\infty\},p,q)$
be the the motivic fundamental torsor of path from $p$ to $q$ on
$\mathbb{P}^{1}\setminus\{0,\infty,1\}$, and ${_{q}\Pi_{p}}$ its
realization with respect to the canonical fiber functor. We will write
$\exp\mathfrak{t}_{3,1}^{0}({\bf k})$ for the set of group-like power
series in ${\bf k}\langle\langle e^{0},e^{1}\rangle\rangle$.\footnote{Later, we will define a Lie algebra $\mathfrak{t}_{n,N}$ and its
sub-Lie algebra $\mathfrak{t}_{n,N}^{0}$ for general $n\geq2$ and
$N\geq1$.} Then ${_{q}\Pi_{p}}({\bf k})$ is canonically isomorphic to $\exp\mathfrak{t}_{3,1}^{0}({\bf k})$
for any $p,q\in\{0,1\}$.\footnote{In this paper, we will use a convention such that the path $\gamma\in\pi_{1}(\mathbb{P}^{1}\setminus\{0,1,\infty\},p,q)$
corresponds to a power series whose coefficients of $e^{a_{k}}\cdots e^{a_{1}}$
is $\int_{0<t_{1}<\cdots<t_{k}<1}\prod_{j=1}^{k}d\log(\gamma(t_{j})-a_{j}).$} For a group-like power series $f\in\exp\mathfrak{t}_{3,1}^{0}({\bf k})$,
we write the corresponding element in ${_{q}\Pi_{p}}({\bf k})$ as
${_{q}f_{p}}$. Then $\mathcal{G}_{1}$ acts on ${_{q}\Pi_{p}}$ for
$p,q\in\{0,1\}$, and these actions can be recovered only from the
data of the map
\[
\lambda_{1}:\mathcal{U}_{1}\xrightarrow{\sigma\mapsto\sigma({_{1}1_{0}})}{_{1}\Pi_{0}}\simeq\exp\mathfrak{t}_{3,1}^{0}
\]
(see \cite[Section 5]{DelGon}). It was proved by Brown \cite{Bro_mix}
that $\lambda_{1}$ is injective, or equivalently the motivic Galois
action $\mathcal{G}_{1}\curvearrowright{_{1}\Pi_{0}}$ is faithful.
On the other hand, Drinfeld \cite{Dr_quasi} introduced an intermediate
closed subscheme ${\rm im}(\lambda_{1})\subset{\rm GRT}_{1}\subset\exp\mathfrak{t}_{3,1}^{0}$
called (the prounipotent version of) the Grothendieck-Teichm\"{u}ller
group, which is closely related to Grothendieck's approach in \cite{Gro}
to the description of the action of the absolute Galois group. The
following is a fundamental conjecture in this area.
\begin{conjecture}
${\rm im}(\lambda_{1})={\rm GRT}_{1}$.
\end{conjecture}

Now let us consider the cyclotomic case. Let $N$ be a positive integer
and $\mu_{N}$ the set of $N$-th roots of unity. Let $\mathcal{MTM}(\mathbb{Z}[\mu_{N},1/N])$
be the Tannakian category of mixed Tate motives over $\mathbb{Z}[\mu_{N},1/N]$,
and $\mathcal{G}_{N}$ its motivic Galois group with respect to the
canonical fiber functor. Then $\mathcal{G}_{N}$ is decomposed as
$\mathcal{G}_{N}=\mathbb{G}_{m}\ltimes\mathcal{U}_{N}$ where $\mathcal{U}_{N}$
is the prounipotent part of $\mathcal{G}_{N}$. For $p,q\in\{0\}\cup\mu_{N}$,
let $\pi_{1}^{{\rm mot}}(\mathbb{P}^{1}\setminus\{0,\infty\}\cup\mu_{N},p,q)$
be the the motivic fundamental torsor of path from $p$ to $q$ on
$\mathbb{P}^{1}\setminus\{0,\infty\}\cup\mu_{N}$, and ${_{q}\Pi_{p}^{(N)}}$
its realization with respect to the canonical fiber functor. We will
write $\exp\mathfrak{t}_{3,N}^{0}({\bf k})$ for the set of group-like
power series in ${\bf k}\langle\langle e^{a}\mid a\in\{0\}\cup\mu_{N}\rangle\rangle$.
Then ${_{q}\Pi_{p}^{(N)}}({\bf k})$ is canonically isomorphic to
$\exp\mathfrak{t}_{3,N}^{0}({\bf k})$ for any $p,q\in\{0,1\}$. For
$f\in\exp\mathfrak{t}_{3,N}^{0}({\bf k})$, we write the corresponding
element in ${_{q}\Pi_{p}^{(N)}}({\bf k})$ as ${_{q}f_{p}}$. Then
$\mathcal{G}_{N}$ acts on ${_{q}\Pi_{p}}$ for all $p,q\in\{0\}\cup\mu_{N}$,
and these actions can be recovered only from the data of the map
\[
\lambda_{N}:\mathcal{U}_{N}\xrightarrow{\sigma\mapsto\sigma({_{1}0_{0}})}{_{1}\Pi_{0}^{(N)}}\simeq\exp\mathfrak{t}_{3,N}^{0}
\]
(see \cite[Section 5]{DelGon}). For general $N>1$, $\lambda_{N}$
is not necessary injective, but it is proved by Deligne \cite{Deli_es}
that $\lambda_{N}$ is injective for $N\in\{2,3,4,8\}$. Furthermore,
Enriquez \cite{Enr_quasi} generalize ${\rm GRT}_{1}$ to an intermediate
closed subscheme ${\rm im}(\lambda_{N})\subset{\rm GRTM}_{(\bar{1},1)}(N)\subset\exp\mathfrak{t}_{3,N}^{0}$
called cyclotomic Grothendieck-Teichm\"{u}ller group of level $N$.
Then the following problem is fundamental.
\begin{problem}
The image ${\rm im}(\lambda_{N})$ is equal to ${\rm GRTM}_{(\bar{1},1)}(N)$
or similar subgroup?
\end{problem}

The following is the main theorem of this paper.
\begin{thm}
\label{thm:main}${\rm im}(\lambda_{2})={\rm GRTM}_{(\bar{1},1)}(2)$.
\end{thm}

\begin{rem}
In \cite{Enr_quasi}, Enriquez also introduce many objects similar
to ${\rm GRTM}_{(\bar{1},1)}(N)$ such like ${\rm GRTM(N)}$, ${\rm GTM}(N)$,
${\rm Pseudo}(N)$, ${\rm GTM}(N)_{l}$, and so on. We make some comments
on the relationships between these objects and Theorem \ref{thm:main}.
First, ${\rm GRTM}(N)$ is an algebraic group $((\mathbb{Z}/N\mathbb{Z})^{\times}\times\mathbb{G}_{m})\ltimes{\rm GRTM}_{(\bar{1},1)}(N)$.
Especially, ${\rm GRTM}(2)=\mathbb{G}_{m}\ltimes{\rm GRTM}_{(\bar{1},1)}(2)$.
Thus Theorem \ref{thm:main} implies ${\rm GRTM}(2)\simeq\mathcal{G}_{2}$.
Second, ${\rm GTM}(N)$ is an algebraic group and ${\rm Pseudo}(N)$
is an $({\rm GRTM}(N),{\rm GTM}(N))$-bitorsor. Then Theorem \ref{thm:main}
implies that ${\rm GTM}(2)\simeq\mathcal{G}_{2}^{{\rm Betti}}$ and
${\rm GTM}(2)\simeq\mathcal{G}_{2}^{{\rm dR},{\rm Betti}}$ where
$\mathcal{G}_{2}^{{\rm Betti}}$ is the motivic Galois group of $\mathcal{MTM}(\mathbb{Z}[1/2])$
with respect to the Betti functor and $\mathcal{G}_{2}^{{\rm dR},{\rm Betti}}$
is the motivic Galois bitorsor of $\mathcal{MTM}(\mathbb{Z}[1/2])$
with respect to the Betti and canonical (=de Rham) functors. Third,
${\rm GTM}(N)_{l}$ is a pro-$l$ variant of ${\rm GTM}(N)$ and there
is natural maps ${\rm GTM}(N)_{l}\to{\rm GTM}(N,\mathbb{Q}_{l})$,
$\widehat{{\rm GT}}\to{\rm GTM}(N)_{l}$, and ${\rm GTM}(N)_{l}\to{\rm GT}_{l}$.
Then the map ${\rm Gal}(\bar{\mathbb{Q}}/\mathbb{Q})\to{\rm GT}_{l}$
studies by Ihara \cite{Ihara} can be factored as
\[
{\rm Gal}(\bar{\mathbb{Q}}/\mathbb{Q})\hookrightarrow\widehat{{\rm GT}}\to{\rm GTM}(N)_{l}\to{\rm GT}_{l}.
\]
Since the canonical map ${\rm Gal}(\bar{\mathbb{Q}}/\mathbb{Q})\to\mathcal{G}_{2}^{\et}(\mathbb{Q}_{l})\simeq\mathcal{G}_{2}^{{\rm Betti}}(\mathbb{Q}_{l})$
is Zariski dense, Theorem \ref{thm:main} implies that the composite
map
\[
{\rm Gal}(\bar{\mathbb{Q}}/\mathbb{Q})\hookrightarrow\widehat{{\rm GT}}\to{\rm GTM}(N)_{l}\to{\rm GTM}(N,\mathbb{Q}_{l})
\]
is Zariski dense for $N=2$.
\end{rem}

The proof of Theorem \ref{thm:main} is essentially based on the result
in \cite{HS_EulerSum}. Let us sketch the proof. First, the dual side
of the morphism
\[
\lambda_{N}:\mathcal{U}\to\exp\mathfrak{t}_{3,N}^{0}
\]
is described as follows. Let $\mathcal{H}_{N}$ be the ring of effective
motivic periods of $\mathcal{MTM}(\mathbb{Z}[\mu_{N},1/N])$ and $(2\pi i)^{\mathfrak{m}}\in\mathcal{H}_{N}$
the motivic $2\pi i$. Then $\mathcal{O}(\mathcal{U})$ is canonically
isomorphic to $\mathcal{H}_{N}/(2\pi i)^{\mathfrak{m}}$. Let us identify
$\mathcal{O}(\exp\mathfrak{t}_{3,N}^{0})$ with $\mathfrak{h}_{N}^{\shuffle}\coloneqq(\mathfrak{h}_{N},\shuffle)$
where $\mathfrak{h}_{N}\coloneqq\mathbb{Q}\langle e_{\xi}\mid\xi\in\{0\}\cup\mu_{N}\rangle$
is a free non-commutative algebra with $N+1$ generators and $\shuffle:\mathfrak{h}_{N}\times\mathfrak{h}_{N}\to\mathfrak{h}_{N}$
is the shuffle product by the $\mathbb{Q}$-linear paring 
\[
\langle\varphi,e_{a_{1}}\cdots e_{a_{k}}\rangle=(\text{coefficient of }e^{a_{k}}\cdots e^{a_{1}}\text{ in }\varphi)\qquad(\varphi\in\exp\mathfrak{t}_{3,N}^{0},e_{a_{1}}\cdots e_{a_{k}}\in\mathfrak{h}_{N}).
\]
Then the dual side of $\lambda_{N}$ is given in terms of motivic
iterated integral as
\[
\mathfrak{h}_{N}^{\shuffle}\xrightarrow{L^{\mathfrak{m}}}\mathcal{H}_{N}\xrightarrow{\bmod\,(2\pi i)^{\mathfrak{m}}}\mathcal{H}_{N}/(2\pi i)^{\mathfrak{m}}
\]
where $L^{\mathfrak{m}}$ is a $\mathbb{Q}$-linear map defined by
the motivic iterated integrals from $0$ to $1$. We write this composite
map as $L^{\mathfrak{a}}:\mathfrak{h}_{N}^{\shuffle}\to\mathcal{H}_{N}/(2\pi i)^{\mathfrak{m}}$.
Then, for $\mathbb{Q}$-algebra ${\bf k}$, the set of ${\bf k}$-rational
points of ${\rm im}(\lambda_{N})$ is given by
\[
\{\varphi\in\exp\mathfrak{t}_{3,N}^{0}({\bf k})\mid\langle\varphi,u\rangle=0\ \text{for all }u\in\ker L^{\mathfrak{a}}\}.
\]
Thus the theorem is equivalent to
\begin{equation}
\langle\varphi,u\rangle=0\quad\quad(\varphi\in{\rm GRTM}_{(\bar{1},1)}(2,{\bf k}),\,u\in\ker L^{\mathfrak{a}}).\label{eq:state1}
\end{equation}
In \cite{HS_EulerSum}, we introduce the set of (level two) confluence
relations $\mathcal{I}_{{\rm CF}}\subset\mathfrak{h}_{2}^{\shuffle}$,
and showed that $\ker L^{\mathfrak{m}}$ is equal to $\mathcal{I}_{{\rm CF}}+(e_{0},e_{1})$
where $(e_{0},e_{1})$ is the ideal of $\mathfrak{h}_{N}^{\shuffle}$
spanned by $e_{0}$ and $e_{1}$. Therefore the statement (\ref{eq:state1})
is equivalent to
\begin{equation}
\langle\varphi,u\rangle=0\quad\quad(\varphi\in{\rm GRTM}_{(\bar{1},1)}(2,{\bf k}),\,u\in\mathcal{I}_{{\rm CF}}).\label{eq:state2}
\end{equation}
Let $\mathrm{PENT}(N,{\bf k})\subset\exp\mathfrak{t}_{3,N}^{0}({\bf k})$
be the set of group-like power series satisfying the mixed pentagon
equation which is one of the defining equations of ${\rm GRTM}_{(\bar{1},1)}(N,{\bf k})$.
The mixed pentagon equation is an equation in $U\mathfrak{t}_{4,N}^{0}({\bf k})=U\mathfrak{t}_{4,N}^{0}(\mathbb{Q})\hat{\otimes}{\bf k}$.
In this paper, we prove the following stronger statement:
\begin{thm}
\label{thm:Gpent-ICF}If $\varphi\in\mathrm{PENT}(2,{\bf k})$ then
$\langle\varphi,u\rangle=0$ for $u\in\mathcal{I}_{{\rm CF}}$.
\end{thm}

Since ${\rm GRTM}_{(\bar{1},1)}(N,{\bf k})\subset\mathrm{PENT}(N,{\bf k})$,
Theorem \ref{thm:Gpent-ICF} implies (\ref{eq:state2}) and so Theorem
\ref{thm:main}.

The remainder of this paper is organized as follows. In Section \ref{sec:Cyc-GT-group},
we recall the definition of cyclotomic Grothendieck-Teichm\"{u}ller
group ${\rm GRTM}_{(\bar{1},1)}(N)$, and give a precise definition
of $\mathrm{PENT}(N,{\bf k})$. In Section \ref{sec:confluence},
we review the definition of confluence relations given in \cite{HS_EulerSum}.
In Section \ref{sec:Constr-elem-Homk}, we give an explicit way to
construct elements of the dual space of $U\mathfrak{t}_{4,N}^{0}(\mathbb{Q})$
that will be used in the subsequent sections. In Section \ref{sec:Broadhurst-duality},
we prove Broadhurst-duality relations for elements of ${\rm GRTM}_{(\bar{1},1)}(2)$.
Finally, In Section \ref{sec:proof-of-main}, we prove Theorems \ref{thm:main}
and \ref{thm:Gpent-ICF}.

\section{\label{sec:Cyc-GT-group}The cyclotomic Grothendieck-Teichm\"{u}ller
group}

In this section, we recall the definition of the cyclotomic Grothendieck-Teichm\"{u}ller
group, which was first introduced in \cite{Enr_quasi}. We basically
follow the notation used in \cite{EnrFur}. Let ${\bf k}$ be a commutative
$\mathbb{Q}$-algebra. For $n\geq2$ and $N\geq1$, let $\mathfrak{t}_{n,N}({\bf k})$
be the completed Lie ${\bf k}$-algebra with generators $t^{1j}$
$(2\leq j\leq n)$, $t(a)^{ij}$ $(i\neq j,\,2\leq i,j\leq n,\,a\in\mathbb{Z}/N\mathbb{Z})$
and relations:
\[
t(a)^{ij}=t(-a)^{ji},\quad\left[t(a)^{ij},t(a+b)^{ik}+t(b)^{jk}\right]=0,\quad\left[t^{1i}+t^{1j}+\sum_{c\in\mathbb{Z}/N\mathbb{Z}}t(c)^{ij},t(a)^{ij}\right]=0,
\]
\[
\left[t^{1i},t^{1j}+\sum_{c\in\mathbb{Z}/N\mathbb{Z}}t(c)^{ij}\right]=0,\quad\left[t^{1i},t(a)^{jk}\right]=0,\left[t^{ij}(a),t^{kl}(b)\right]=0
\]
($i,j,k,l\in\{2,\dots,n\}$ are all distinct and $a,b\in\mathbb{Z}/N\mathbb{Z}$).
Let $\mathfrak{t}_{n,N}^{0}({\bf k})$ be the completed Lie subalgebra
of $t_{n,N}({\bf k})$ with the same generators except $t^{1n}$ and
the relations as $\mathfrak{t}_{n,N}$ (except those involving $t^{1n}$).
Fix a generator $\zeta_{N}$ of $\mu_{N}$. Note that $\mathfrak{t}_{3,N}^{0}({\bf k})$
is the completed free ${\bf k}$-Lie algebra of rank $N+1$ with generators
$e^{0}\coloneqq t^{12}$ and $e^{\zeta_{N}^{a}}\coloneqq t(a)^{23}$
$(a\in\mathbb{Z}/N\mathbb{Z})$. For a partially defined map $f:\{1,\dots,m\}$
to $\{1,\dots,n\}$ satisfying $f(1)=1$, define the continuous Lie
${\bf k}$-algebra morphism $t_{n,N}\to t_{m,N}:x\mapsto x^{f}=x^{f^{-1}(1),\dots,f^{-1}(n)}$
by
\[
(t(a)^{ij})^{f}=\sum_{i'\in f^{-1}(i),j'\in f^{-1}(j)}t(a)^{i'j'}
\]
and
\[
(t^{1j})^{f}=\sum_{j'\in f^{-1}(j)}t^{1j'}+\frac{1}{2}\sum_{j',j''\in f^{-1}(j)}\sum_{c\in\mathbb{Z}/N\mathbb{Z}}t(c)^{j'j''}+\sum_{i'\in f^{-1}(1)\setminus\{1\},j'\in f^{-1}(j)}t(c)^{i'j'}.
\]
Hereafter, we simply write $t^{ij}$ for the generator $t(0)^{ij}$
of $\mathfrak{t}_{n,1}({\bf k})$ if there is no risk of confusion.
For a partially defined map $g:\{2,\dots,m\}\to\{1,\dots,n\}$, define
the continuous Lie ${\bf k}$-algebra morphism $t_{n,1}\to t_{m,N}:x\mapsto x^{g}=x^{g^{-1}(1),\dots,g^{-1}(n)}$
by
\[
(t^{ij})^{g}=\sum_{i'\in g^{-1}(i),j'\in g^{-1}(j)}t(0)^{i'j'}\qquad(i\in\{1,\dots,n\},j\in\{2,\dots,n\}).
\]
We also define the continuous Lie ${\bf k}$-algebra homomorphism
$\delta:\mathfrak{t}_{n,N}^{0}({\bf k})\to\mathfrak{t}_{n,1}^{0}({\bf k})$
by $t^{1j}\mapsto t^{1j}$ and $t^{ij}(a)\mapsto\delta_{a,0}t^{ij}$.
\begin{defn}
We say that the pair $(g,h)\in\exp\mathfrak{t}_{3,1}^{0}({\bf k})\times\exp\mathfrak{t}_{3,N}^{0}({\bf k})$
satisfies the \emph{mixed pentagon equation} if
\[
h^{1,2,34}h^{12,3,4}=g^{2,3,4}h^{1,23,4}h^{1,2,3}
\]
holds in $\exp\mathfrak{t}_{4,N}^{0}$.
\end{defn}

Furthermore, we say that $h\in\exp\mathfrak{t}_{3,1}^{0}({\bf k})$
satisfies the pentagon equation if $(h,h)$ satisfies the pentagon
equation. This coincides the original definition of the pentagon equation.
\begin{defn}[(The prounipotent part of) the Grothendieck-Teichm\"{u}ller group]
We define ${\rm GRT}_{1}({\bf k})$ as the set of $g\in\exp\mathfrak{t}_{3,1}^{0}({\bf k})$
satisfying the following conditions:
\begin{enumerate}
\item $g$ satisfies the mixed pentagon equation,
\item the coefficient of $e^{0}e^{1}$ in $g$ is zero.
\end{enumerate}
\end{defn}

Now, we define ${\rm GRTM}_{(\bar{1},1)}(N,{\bf k})$.

\begin{defn}[(The prounipotent part of) the cyclotomic Grothendieck-Teichm\"{u}ller
group]
We define ${\rm GRTM}_{(\bar{1},1)}(N,{\bf k})$ as the set of $h\in\exp\mathfrak{t}_{3,N}^{0}({\bf k})$
satisfying the following conditions:
\begin{enumerate}
\item There exists $g\in{\rm GRT}_{1}({\bf k})$ such that $(g,h)$ satisfies
the mixed pentagon equation,
\item the coefficient of $e^{1}$ in $h$ is zero,
\item $\tau_{\zeta_{N}}(h)^{-1}\cdot\tau_{\zeta_{N}}(\sigma(h))\cdot\sigma(h)^{-1}\cdot h=1$,
\item $e^{0}+\sum_{a\in\mathbb{Z}/N\mathbb{Z}}{\rm Ad}(\tau_{a}h^{-1})(e^{\zeta_{N}^{a}})+{\rm Ad}\Big(h^{-1}\cdot\sigma(h)\Big)(e^{\infty})=0$,
\end{enumerate}
where $e^{0}+\sum_{\eta\in\mu_{N}}e^{\eta}+e^{\infty}=0$ and $\tau_{\eta}$
($\eta\in\mu_{N}$) (resp. $\sigma$) is the automorphism defined
by $e^{a}\mapsto e^{\eta a}$ (resp. $e^{a}\mapsto e^{a^{-1}}$) for
$a\in\{0\}\cup\mu_{N}$.

\end{defn}

\begin{rem}
The element $g\in{\rm GRT}_{1}({\bf k})$ in (1) is uniquely determined
from $h$. In fact, $g$ is equal to $\delta(h)$ since the $h^{1,2,34}h^{12,3,4}$
and $g^{2,3,4}h^{1,23,4}h^{1,2,3}$ are mapped to $\delta(h)$ and
$g$ respectively by the map $t^{1j}\mapsto0$, $t^{ij}(a)\mapsto\delta_{a,0}t^{i-1,j-1}$
by (2).
\end{rem}

\begin{rem}
In \cite{EnrFur}, ${\rm GRTM}_{(\bar{1},1)}(N,{\bf k})$ is defined
as the set of pairs $(g,h)\in\exp\mathfrak{t}_{3,1}^{0}({\bf k})\times\exp\mathfrak{t}_{3,N}^{0}({\bf k})$
satisfying the same condition.
\end{rem}

\begin{defn}
We define $\mathrm{PENT}(N,{\bf k})$ as the set of $h\in\exp\mathfrak{t}_{3,N}^{0}({\bf k})$
such that $(\delta(h),h)$ satisfies the mixed pentagon equation.
\end{defn}

The functors ${\bf k}\mapsto\exp\mathfrak{t}_{3,N}^{0}({\bf k})$,
${\bf k}\mapsto{\rm GRT}_{1}({\bf k})$, ${\bf k}\mapsto{\rm GRTM}_{(\bar{1},1)}(N,{\bf k})$
and ${\bf k}\mapsto\mathrm{PENT}(N,{\bf k})$ from the category of
commutative $\mathbb{Q}$-algebras to the category of sets are representable.
We denote the affine $\mathbb{Q}$-schemes corresponding these functors
by $\exp\mathfrak{t}_{3,N}^{0}$, ${\rm GRT}_{1}$, ${\rm GRTM}_{(\bar{1},1)}(N)$
and $\mathrm{PENT}(N)$, respectively.
\begin{prop}
\label{prop:pent_property}Let $h=h(e^{0},e^{\zeta_{N}^{0}},\dots,e^{\zeta_{N}^{N-1}})\in{\rm PENT}(N,{\bf k})$.
Then,
\begin{enumerate}[label=\textup{(\roman{enumi})}]
\item The coefficients of $e^{0}$ and $e^{1}$ in $h$ are zero,
\item $g=\delta(h)$ satisfies $g^{1,2,34}g^{12,3,4}=g^{2,3,4}g^{1,23,4}g^{1,2,3}$,
\item $g=\delta(h)$ satisfies the duality relation $g(e_{0},e_{1})g(e_{1},e_{0})=1$
and regularized double shuffle relation (see \cite{Fur_dsh}).
\item $h$ satisfies the (special case of) distribution relation 
\[
h(Ne^{0},e^{1},\dots,e^{1})=\exp(\rho e^{1})h(e^{0},e^{1},0,\dots,0)
\]
where $\rho$ is the coefficient of $e^{1}$ in $h(Ne^{0},e^{1},\dots,e^{1})-h(e^{0},e^{1},0,\dots,0)$.
\end{enumerate}
\end{prop}

\begin{proof}
By definition,
\begin{equation}
h^{1,2,34}h^{12,3,4}=\delta(h)^{2,3,4}h^{1,23,4}h^{1,2,3}.\label{eq:pent}
\end{equation}
The coefficient of $t^{12}$ in $h^{1,2,34}h^{12,3,4}$ (resp. $\delta(h)^{2,3,4}h^{1,23,4}h^{1,2,3}$)
is equal to the $\langle h,e_{0}\rangle$ (resp. $2\langle h,e_{0}\rangle$).
Thus $\langle h,e_{0}\rangle=0$. The coefficient of $t^{34}(0)$
in $h^{1,2,34}h^{12,3,4}$ (resp. $\delta(h)^{2,3,4}h^{1,23,4}h^{1,2,3}$)
is equal to the $\langle h,e_{1}\rangle$ (resp. $2\langle h,e_{1}\rangle$).
Thus $\langle h,e_{1}\rangle=0$. Hence (i) is proved. Put $g=\delta(h)$.
By applying the map
\[
\mathfrak{t}_{n,N}\to\mathfrak{t}_{n,1}\quad;\quad t^{1j}\mapsto t^{1j},\ t^{ij}(a)\mapsto\delta_{a,0}t^{ij}
\]
to \ref{eq:pent}, we have
\[
g^{1,2,34}g^{12,3,4}=g^{2,3,4}g(t^{12}+t^{13}+t^{23},t^{24}+t^{34})g^{1,2,3}.
\]
Furthermore, since $\langle g,e_{0}\rangle=0$, $g\in\exp\mathfrak{t}_{3,1}^{0}({\bf k})$,
and $[t^{23},t^{12}+t^{13}]=[t^{23},t^{24}+t^{34}]=0$, we have
\[
g(t^{12}+t^{13}+t^{23},t^{24}+t^{34})=g^{1,23,4}.
\]
Thus (ii) is also proved. Then, (iii) follows from (ii) and the results
in \cite{Fur_hex} and \cite{Fur_dsh}. Finally (iv) is proved in
\cite[Proposition 2.9 (2)]{EnrFur}
\end{proof}

\section{\label{sec:confluence}Review of confluence relation}

In this section, we recall the definition of the set of confluence
relations $\mathcal{I}_{{\rm CF}}$ in \cite{HS_EulerSum}. We basically
follow the same notation as \cite{HS_EulerSum}, except that we use
$\mathbb{Q}$-modules here instead of $\mathbb{Z}$-modules. For $S\subset\{0,1,-1\}$,
we denote by $\mathcal{A}(S)$ the free $\mathbb{Q}$-algebra generated
by formal symbols $\{e_{a}\mid a\in S\}$. Furthermore, we define
the subspace $\mathcal{A}^{0}(S)\subset\mathcal{A}(S)$ and $\mathcal{C}\subset\mathcal{A}(\{0,1\})$
by 
\begin{align*}
\mathcal{A}^{0}(S) & \coloneqq\mathbb{Q}\oplus\bigoplus_{k=1}^{\infty}\bigoplus_{\substack{a_{1},\dots,a_{k}\in S\\
a_{1}\neq0,a_{k}\neq1
}
}\mathbb{Q}e_{a_{1}}\cdots e_{a_{k}},\\
\mathcal{C} & \coloneqq\mathbb{Q}\oplus e_{1}\mathbb{Q}\left\langle e_{0},e_{1}\right\rangle .
\end{align*}
Put 
\[
\mathcal{B}\coloneqq\sum_{\substack{a_{1},\dots,a_{k}\in\{0,-1,z,-z^{2}\}\\
a_{1}\neq0,a_{k}\neq z
}
}\mathbb{Q}e_{a_{1}}\cdots e_{a_{k}}.
\]
Let ${\rm reg}_{\shuffle}:\mathcal{A}(\{0,1\})\to\mathcal{A}^{0}(\{0,1\})$
be the unique $\shuffle$-morphism such that ${\rm reg}_{\shuffle}(e_{0})={\rm reg}_{\shuffle}(e_{1})=0$
and ${\rm reg}_{\shuffle}(u)=u$ for $u\in\mathcal{A}^{0}(\{0,1\})$.
Define an automorphism $\varrho$ of $\mathcal{A}(\{0,1\})$ by $\varrho(e_{0})=e_{0}-e_{1}$
and $\varrho(e_{1})=-e_{1}$. Furthermore, define a $\mathbb{Q}$-algebra
homomorphism ${\rm dist}:\mathcal{A}(\{0,1\})\to\mathcal{A}(\{0,1,-1\})$
by ${\rm dist}(e_{0})=2e_{0}$ and ${\rm dist}(e_{1})=e_{1}+e_{-1}$.
\begin{defn}
For $d\geq0$ and $\Bbbk=(k_{1},\dots,k_{d})\in\mathbb{Z}_{\geq1}^{d}$,
let $\theta(\Bbbk)$, $\theta_{1}(\Bbbk)$, $\theta'(\Bbbk)$ and
${\rm w}^{\star}(\Bbbk)$ be elements of $\mathcal{A}(\{0,1,-1\})$
defined by

\begin{align*}
\theta(k_{1},\dots,k_{d}) & \coloneqq\begin{cases}
-\sum_{j=1}^{d}e_{1}e_{0}^{k_{j}-1}\cdots e_{1}e_{0}^{k_{1}-1}\shuffle\theta(k_{j+1},\dots,k_{d}) & d>0\\
1 & d=0,
\end{cases}
\end{align*}
\[
\theta'(\Bbbk)\coloneqq{\rm dist}\circ{\rm reg}_{\shuffle}\circ\varsigma\circ\theta(\Bbbk),
\]
\[
\theta_{1}(\Bbbk)\coloneqq\theta'(\Bbbk)+2\sum_{\substack{\Bbbk=(\Bbbk',\{1\}^{2m})\\
m\geq1
}
}e_{-1}e_{0}^{2m-1}\shuffle\theta'(\Bbbk'),
\]
\[
{\rm w}^{\star}(k_{1},\dots,k_{d})\coloneqq\begin{cases}
-e_{1}e_{0}^{k_{1}-1}(e_{0}-e_{1})e_{0}^{k_{2}-1}\cdots(e_{0}-e_{1})e_{0}^{k_{d}-1} & d>0\\
1 & d=0.
\end{cases}
\]
Note that $\{\varrho({\rm w}^{\star}(\Bbbk))\}$ is a basis of $\mathcal{C}$.
Now, by using the above notations, we define a $\mathbb{Q}$-linear
map $\wp:\mathcal{C}\to\mathcal{A}^{0}(\{0,1,-1\})$ by
\[
\wp(\varrho({\rm w}^{\star}(k_{1},\dots,k_{d})))\coloneqq\theta_{1}(k_{1},\dots,k_{d}).
\]

In \cite[Lemma 2.17]{HS_EulerSum}, the author proved that
\begin{equation}
L^{\mathfrak{m}}({\rm reg}_{\shuffle}(\varsigma(w))=L^{\mathfrak{m}}(\wp(w))\label{eq:Lm_eq_wp}
\end{equation}
for $w\in\mathcal{C}$. The properties of $L^{\mathfrak{m}}$ used
in the proof of \cite[Lemma 2.17]{HS_EulerSum} is the regularized
double shuffle relations and duality relation for $L^{\mathfrak{m}}\mid_{\mathcal{A}^{0}(\{0,1\})}$,
and (the special case of) the distribution relation $(1-2^{1-2m})L^{\mathfrak{m}}(e_{1}e_{0}^{2m-1})=-L^{\mathfrak{m}}(e_{-1}e_{0}^{2m-1})$.
These properties are also satisfied by 
\[
Z_{\varphi}:\mathcal{A}^{0}(\{0,1,-1\})\to{\bf k}\quad;\quad u\mapsto\langle\varphi,u\rangle
\]
for $\varphi\in{\rm PENT}(2,{\bf k})$ by Proposition \ref{prop:pent_property}.
Thus (\ref{eq:Lm_eq_wp}) also holds for $Z_{\varphi}$, i.e., we
have
\begin{equation}
\langle\varphi,{\rm reg}_{\shuffle}(\varsigma(w)\rangle=\langle\varphi,\wp(w)\rangle\label{eq:phico_eq_wp}
\end{equation}
for $\varphi\in{\rm PENT}(2,{\bf k})$ and $w\in\mathcal{C}$.
\end{defn}

Define ${\rm reg}_{z\to0}:\mathcal{B}\to\mathcal{A}^{0}(\{0,1,-1\})$
by
\[
\mathcal{B}\xrightarrow{e_{-1}\mapsto0}\mathcal{B}'\xrightarrow[\simeq]{u\shuffle v\mapsto u\otimes v}\mathcal{B}''\otimes\mathcal{B}'''\xrightarrow{\overline{{\rm reg}}_{z\to0}\otimes\epsilon}\mathcal{A}^{0}(\{0,1\})\otimes\mathcal{C}\xrightarrow{u\otimes v\mapsto{\rm dist}(u)\shuffle\wp(v)}\mathcal{A}^{0}(\{0,1,-1\}).
\]
where $\mathcal{B}'$, $\mathcal{B}''$, $\mathcal{B}'''$ and $\mathcal{C}$
are $\mathbb{Q}$-modules defined by

\begin{align*}
\mathcal{B}' & \coloneqq\mathcal{B}\cap\mathbb{Q}\left\langle e_{0},e_{z},e_{-z^{2}}\right\rangle ,\\
\mathcal{B}'' & \coloneqq\mathcal{B}'\cap\left(\mathbb{Q}\oplus e_{z}\mathbb{Q}\left\langle e_{0},e_{z},e_{-z^{2}}\right\rangle \right),\\
\mathcal{B}''' & \coloneqq\mathcal{B}'\cap\mathbb{Q}\left\langle e_{0},e_{-z^{2}}\right\rangle ,
\end{align*}
$\overline{{\rm reg}}_{z\to0}$ is the ring homomorphism from $\mathbb{Q}\langle e_{0},e_{z},e_{-z^{2}}\rangle$
to $\mathcal{A}(\{0,1\})$ defined by $\overline{{\rm reg}}_{z\to0}(e_{0})=\overline{{\rm reg}}_{z\to0}(e_{-z^{2}})=e_{0}$
and $\overline{{\rm reg}}_{z\to0}(e_{z})=e_{1}$, and $\epsilon$
is the ring homomorphism from $\mathbb{Q}\langle e_{0},e_{-z^{2}}\rangle$
to $\mathcal{A}(\{0,1\})$ defined by $\epsilon(e_{0})=e_{0}$ and
$\epsilon(e_{-z^{2}})=e_{1}$.

For $c\in\{0,1,-1\}$, define $\partial_{c}:\mathcal{B}\to\mathcal{B}$
by
\[
\partial_{c}(e_{a_{1}}\cdots e_{a_{k}})=\sum_{i=1}^{k}\left({\rm ord}_{z=c}(a_{i}-a_{i+1})-{\rm ord}_{z=c}(a_{i}-a_{i-1})\right)e_{a_{1}}\cdots e_{a_{i-1}}e_{a_{i+1}}\cdots e_{a_{k}}
\]
where $a_{0}=0$, $a_{k+1}=z$, and ${\rm ord}_{z=c}(0)=0$. Define
$\varphi:\mathcal{B}\to\mathcal{A}^{0}(\{0,1,-1\})$ by
\[
\varphi(u)=\sum_{l=0}^{\infty}\sum_{c_{1},\dots,c_{l}\in\{0,1,-1\}}{\rm reg}_{z\to0}(\partial_{c_{1}}\cdots\partial_{c_{l}}u)\shuffle{\rm reg}_{\shuffle}(e_{c_{1}}\cdots e_{c_{l}}).
\]
Now, $\mathcal{I}_{{\rm CF}}\subset\mathcal{A}^{0}(\{0,1,-1\})$ is
defined by
\[
\mathcal{I}_{{\rm CF}}\coloneqq\{u\mid_{z\to1}-\varphi(u)\mid u\in\mathcal{B}\}
\]
where the map $u\mapsto u\mid_{z\to1}$ is the ring morphism from
$\mathbb{Q}\left\langle e_{0},e_{-1},e_{z},e_{-z^{2}}\right\rangle $
to $\mathcal{A}(\{0,1,-1\})$ defined by $e_{0}\mid_{z\to1}=e_{0}$,
$e_{-1}\mid_{z\to1}=e_{-z^{2}}\mid_{z\to1}=e_{-1}$, and $e_{z}\mid_{z\to1}=e_{1}$.
The following is the main theorem of \cite{HS_EulerSum}.
\begin{thm}
\label{thm:kerLm=00003DICF}Let $L^{\mathfrak{m}}:\mathcal{A}^{0}(\{0,1,-1\})\to\mathcal{H}_{2}$
be the $\mathbb{Q}$-linear map defined by $L^{\mathfrak{m}}(e_{a_{1}}\cdots e_{a_{k}})=I^{\mathfrak{m}}(0;a_{1},\dots,a_{k};1)$.
Then
\[
\ker(L^{\mathfrak{m}})=\mathcal{I}_{{\rm CF}}.
\]
\end{thm}

\section{\label{sec:Constr-elem-Homk}The dual vector space of $U\mathfrak{t}_{4,N}(\mathbb{Q})$}

Let $U\mathfrak{t}_{n,N}({\bf k})$ be the completion of the universal
enveloping algebra of $\mathfrak{t}_{n,N}({\bf k})$. Recall that
the mixed pentagon equation is formulated as an equality in $\exp\mathfrak{t}_{4,N}({\bf k})\subset U\mathfrak{t}_{4,N}({\bf k})=U\mathfrak{t}_{4,N}(\mathbb{Q})\hat{\otimes}{\bf k}$.
Let $R_{n,N}$ be the ring of formal power series over $\mathbb{Q}$
in free variables $\tilde{t}^{ij}$ $(2\leq j\leq n)$ and $\tilde{t}(a)^{ij}$
$(i\neq j,\,2\leq i,j\leq n,\,a\in\mathbb{Z}/N\mathbb{Z})$. Then
we can regard $U\mathfrak{t}_{n,N}(\mathbb{Q})$ as a quotient of
$R_{n,N}$ by a surjective map
\[
R_{n,N}\to U\mathfrak{t}_{n,N}(\mathbb{Q})\quad;\quad\tilde{t}^{1j}\mapsto t^{1j},\ \tilde{t}(a)^{ij}\mapsto t(a)^{ij}.
\]
By an obvious way, the (topological) dual vector space ${\rm Hom}_{\mathbb{Q}}^{{\rm cont}}(R_{n,N}(\mathbb{Q}),\mathbb{Q})$
of $R_{n,N}(\mathbb{Q})$ can be regarded as the non-commutative polynomial
ring over $\mathbb{Q}$ in $n-1+(n-1)(n-2)N$ variables. Then the
dual space of $U\mathfrak{t}_{4,N}$ is embedded as the $\mathbb{Q}$-vecor
subspace of this non-commutative polynomial ring. In this section,
we propose a way to construct elements of this subspace.

\subsection{Definition of $\Omega_{F}^{\star}$}

Let $F$ be a field, 
\[
\Omega_{F}^{\bullet}\coloneqq\left(\Omega_{F}^{0}\to\Omega_{F}^{1}\to\Omega_{F}^{2}\to\cdots\right)
\]
the complex of K\"{a}hler differential forms on ${\rm Spec}(F)$,
\[
B^{\bullet}(\Omega_{F}^{\bullet})\coloneqq(B^{0}(\Omega_{F}^{\bullet})\to B^{1}(\Omega_{F}^{\bullet})\to\cdots)
\]
the reduce bar complex of $\Omega_{F}^{\bullet}$, and
\[
H^{0}(B^{\bullet}(\Omega_{F}^{\bullet}))\coloneqq\ker(B^{0}(\Omega_{F}^{\bullet})\to B^{1}(\Omega_{F}^{\bullet}))
\]
the first cohomology group of $B^{\bullet}(\Omega_{F}^{\bullet})$.
Define a $\mathbb{Q}$-linear subspace $\Omega_{F}^{\star}$ of $T(\Omega_{F}^{1})\coloneqq\bigoplus_{k=0}^{\infty}(\Omega_{F}^{1})^{\otimes k}$
by
\[
\Omega_{F}^{\star}\coloneqq T(\bar{\Omega}_{F})\cap H^{0}(B^{\bullet}(\Omega_{F}^{\bullet}))
\]
where
\[
\bar{\Omega}_{F}^{1}\coloneqq\{\frac{da}{a}\mid a\in F^{\times}\}\subset\Omega_{F}^{1}.
\]
More explicitly, $\Omega_{F}^{\star}$ is the kernel of the $\mathbb{Q}$-linear
map from $\bigoplus_{k=0}^{\infty}(\bar{\Omega}_{F}^{1})^{\otimes k}$
to $\bigoplus_{i,j=0}^{\infty}(\bar{\Omega}_{F}^{1})^{\otimes i}\otimes\Omega_{F}^{2}\otimes(\bar{\Omega}_{F}^{1})^{\otimes j}$
defined by
\[
\omega_{1}\otimes\cdots\otimes\omega_{k}\mapsto\sum_{i=1}^{k-1}\omega_{1}\otimes\cdots\omega_{i-1}\otimes(\omega_{i}\wedge\omega_{i+1})\otimes\omega_{i+2}\otimes\cdots\otimes\omega_{k}.
\]

\subsection{The characterization of $U\mathfrak{t}_{n,N}$}

Put $\mathbb{Q}_{N}\coloneqq\mathbb{Q}(\mu_{N})$ and $F_{n,N}\coloneqq\mathbb{Q}_{N}(z_{2},\dots,z_{n})$.
Fix a generator $\zeta_{N}\in\mu_{N}$. Let $\Omega_{n,N}^{1}$ the
$\mathbb{Q}$-submodule of $\Omega_{F_{n,N}}^{1}$ spanned by
\[
\frac{dz_{i}}{z_{i}}\ \ \ (2\leq i\leq n)
\]
and
\[
\frac{dz_{i}}{z_{i}-\zeta_{N}^{a}z_{j}}\ \ \ (2\leq i,j\leq n,\ i\neq j,\ a\in\mathbb{Z}/N\mathbb{Z}).
\]
We regard $T(\Omega_{n,N}^{1})\coloneqq\bigoplus_{k=0}^{\infty}\Omega_{n,N}^{\otimes k}\subset T(\Omega_{F_{n,N}}^{1})$
as the $\mathbb{Q}$-algebra by $u\cdot v=v\otimes u$.\footnote{The standard definition of the multiplication is $u\cdot v=u\otimes v$,
but, we change the order here to reconcile the differences in the
orders of multiplications used in the definitions of associator and
$\mathcal{I}_{{\rm CF}}$.} Then we identity $R_{n,N}({\bf k})$ with the dual vector space of
$T(\Omega_{n,N}^{1})$ by an element
\[
\sum_{m=0}^{\infty}\left(\sum_{i}\frac{dz_{i}}{z_{i}}\cdot s^{1i}+\sum_{i,j,a}\frac{dz_{i}}{z_{i}-\zeta_{N}^{a}z_{j}}\cdot s(a)^{ij}\right)^{m}\in T(\Omega_{n,N})\hat{\otimes}R_{n,N}.
\]
Now, we can give a characterization of the dual vector space of $U\mathfrak{t}_{n,N}({\bf k})$.
\begin{prop}
\label{prop:dual_Ut}For $n\geq2$ and $N\geq1$, $U\mathfrak{t}_{n,N}(\mathbb{Q})$
is the dual vector space of 
\[
\Omega_{F_{n,N}}^{\star}\cap T(\Omega_{n,N}^{1})\quad(\subset T(\Omega_{n,N}^{1})),
\]
i.e.,
\[
U\mathfrak{t}_{n,N}({\bf k})={\rm Hom}_{\mathbb{Q}}(\Omega_{F_{n,N}}^{\star}\cap T(\Omega_{n,N}^{1}),{\bf k}),\quad\Omega_{F_{n,N}}^{\star}\cap T(\Omega_{n,N}^{1})={\rm Hom}_{\mathbb{Q}}^{{\rm cont}}(U\mathfrak{t}_{n,N}(\mathbb{Q}),\mathbb{Q}).
\]
\end{prop}

\begin{proof}
It follows from definition. 
\end{proof}

\subsection{Construction of elements of $\Omega_{F}^{\star}$}

For $k\geq0$, we denote by $A_{F}^{k}$ the $\mathbb{Q}$-vector
space generated by formal symbols $\mathbb{I}(a_{0};a_{1},\dots,a_{k};a_{k+1})$
with $a_{1},\dots,a_{k}\in F$ and $a_{0},a_{k}\in F\cup\{\infty\}$.
Let us define a $\mathbb{Q}$-linear map $\partial:A_{F}^{k}\to A_{F}^{k-1}\otimes\bar{\Omega}_{F}^{1}$
by
\[
\partial\mathbb{I}(a_{0};a_{1},\dots,a_{k};a_{k+1})=\sum_{r\in\{\pm1\}}r\sum_{i=1}^{k}\mathbb{I}(a_{0};a_{1},\dots,\widehat{a_{i}},\dots,a_{k};a_{k+1})\otimes d\log(a_{i+r}-a_{i}),
\]
where we put $d\log(a)=0$ for $a\in\{0,\infty\}$.
\begin{defn}
For $k\geq0$, we define a $\mathbb{Q}$-linear map $\psi^{k}:A_{F}^{k}\to(\bar{\Omega}_{F}^{1})^{\otimes k}$
as follows. For the case $k=0$, we put $\psi^{0}(\mathbb{I}(a_{0};a_{1}))=1$.
For $k\geq1$, define $\psi^{k}$ recursively as the composite map
\[
A_{F}^{k}\xrightarrow{\partial}A_{F}^{k-1}\otimes\bar{\Omega}_{F}^{1}\xrightarrow{\psi^{k-1}\otimes{\rm id}}(\bar{\Omega}_{F}^{1})^{\otimes(k-1)}\otimes\bar{\Omega}_{F}^{1}\simeq(\bar{\Omega}_{F}^{1})^{\otimes k}.
\]
\end{defn}

\begin{lem}
\label{lem:der-wedge-is-zero}The composite map 
\[
A_{F}^{k}\xrightarrow{\partial}A_{F}^{k-1}\otimes\bar{\Omega}_{F}^{1}\xrightarrow{\partial\otimes{\rm id}}A_{F}^{k-2}\otimes\bar{\Omega}_{F}^{1}\otimes\bar{\Omega}_{F}^{1}\xrightarrow{u\otimes\omega_{1}\otimes\omega_{2}\mapsto u\otimes(\omega_{1}\wedge\omega_{2})}A_{F}^{k-2}\otimes\Omega_{F}^{2}
\]
 is zero.
\end{lem}

\begin{proof}
By definition, we have
\begin{align*}
 & (\partial\otimes{\rm id})\circ\partial(\mathbb{I}(a_{0};a_{1},\dots,a_{k};a_{k+1}))\\
 & =\sum_{1\leq i\leq k}\sum_{r\in\{\pm1\}}\partial(\mathbb{I}(a_{0};a_{1},\dots,\widehat{a_{i}},\dots,a_{k};a_{k+1}))\otimes d\log(a_{i+r}-a_{i})\\
 & =\sum_{1\leq j<i\leq k}\sum_{r,s\in\{\pm1\}}rs\partial(\mathbb{I}(a_{0};a_{1},\dots,\widehat{a_{j}},\dots,\widehat{a_{i}},\dots,a_{k};a_{k+1}))\\
 & \quad\quad\quad\otimes\begin{cases}
d\log(a_{j+s}-a_{j})\otimes d\log(a_{i+r}-a_{i}) & (j,s)\neq(i-1,1)\\
d\log(a_{i+1}-a_{i-1})\otimes d\log(a_{i+r}-a_{i}) & (j,s)=(i-1,1)
\end{cases}\\
 & \quad+\sum_{1\leq i<j\leq k}\sum_{r,s\in\{\pm1\}}\partial(\mathbb{I}(a_{0};a_{1},\dots,\widehat{a_{i}},\dots,\widehat{a_{j}},\dots,a_{k};a_{k+1}))\\
 & \quad\quad\quad\otimes\begin{cases}
d\log(a_{j+s}-a_{j})\otimes d\log(a_{i+r}-a_{i}) & (j,s)\neq(i+1,-1)\\
d\log(a_{i+1}-a_{i-1})\otimes d\log(a_{i+r}-a_{i}) & (j,s)=(i+1,-1)
\end{cases}\\
 & =\sum_{1\leq i<j\leq k}\sum_{r,s\in\{\pm1\}}rs\partial(\mathbb{I}(a_{0};a_{1},\dots,\widehat{a_{i}},\dots,\widehat{a_{j}},\dots,a_{k};a_{k+1}))\\
 & \qquad\qquad\otimes\left(d\log(a_{i+r}-a_{i})\otimes d\log(a_{j+s}-a_{j})+d\log(a_{j+s}-a_{j})\otimes d\log(a_{i+r}-a_{i})\right)\\
 & \quad+\sum_{1\leq i<k}\partial(\mathbb{I}(a_{0};a_{1},\dots,\widehat{a_{i},a_{i+1}},\dots,a_{k};a_{k+1}))\\
 & \quad\quad\quad\otimes\Bigg(\sum_{r\in\{\pm1\}}r\bigg(d\log(a_{i+2}-a_{i})-d\log(a_{i+1}-a_{i})\bigg)\otimes d\log(a_{i+1+r}-a_{i+1})\\
 & \quad\quad\quad\quad\quad-\sum_{r\in\{\pm1\}}r\bigg(d\log(a_{i+1}-a_{i-1})-d\log(a_{i+1}-a_{i})\bigg)\otimes d\log(a_{i+r}-a_{i})\Bigg).
\end{align*}
Since
\begin{align*}
d\log(a_{i+r}-a_{i})\wedge d\log(a_{j+s}-a_{j})-d\log(a_{j+s}-a_{j})\wedge d\log(a_{i+r}-a_{i}) & =0,\\
\sum_{r\in\{\pm1\}}r\bigg(d\log(a_{i+2}-a_{i})-d\log(a_{i+1}-a_{i})\bigg)\wedge d\log(a_{i+1+r}-a_{i+1}) & =0,\\
\sum_{r\in\{\pm1\}}r\bigg(d\log(a_{i+1}-a_{i-1})-d\log(a_{i+1}-a_{i})\bigg)\wedge d\log(a_{i+r}-a_{i}) & =0,
\end{align*}
we obtain the claim.
\end{proof}
We put $A_{F}=\bigoplus_{k=0}^{\infty}A_{F}^{k}$ and $\psi=\bigoplus_{k=0}^{\infty}\psi_{k}:A_{F}\to T(\bar{\Omega}_{F}^{1})$.

\begin{prop}
\label{prop:psi-in-OmegaStar}For $u\in A_{F}$, we have $\psi(u)\in\Omega_{F}^{\star}$.
\end{prop}

\begin{proof}
It follows from Lemma \ref{lem:der-wedge-is-zero}.
\end{proof}
Recall that the dual vector space of $U\mathfrak{t}_{n,N}(\mathbb{Q})$
is given by $\Omega_{F_{n,N}}^{\star}\cap T(\Omega_{n,N}^{1})$. By
Proposition \ref{prop:psi-in-OmegaStar}, we can construct elements
of this space. Let $A_{n,N}$ be a subspace of $A_{F_{n,N}}$ spanned
by the elements $\mathbb{I}(a_{0};a_{1},\dots,a_{k};a_{k+1})$ such
that $d\log(a_{i}-a_{j})\in\Omega_{n,N}^{1}$ for all $i,j\in\{0,\dots,k+1\}$.
\begin{prop}
\label{prop:psi-in-Utn-dual}For $u\in A_{n,N}$, we have
\[
\psi(u)\in\Omega_{F_{n,N}}^{\star}\cap T(\Omega_{n,N}^{1})\quad\quad(={\rm Hom}_{\mathbb{Q}}^{{\rm cont}}(U\mathfrak{t}_{n,N}(\mathbb{Q}),\mathbb{Q})\,).
\]
\end{prop}

\begin{proof}
It is just a special case of Proposition \ref{prop:psi-in-OmegaStar}.
\end{proof}
Let
\[
\mathfrak{t}_{n,N}^{(1)}\coloneqq\bigoplus_{j=2}^{n}\mathbb{Q}t^{1j}\oplus\bigoplus_{2\leq i<j\leq n,\,a\in\mathbb{Z}/N\mathbb{Z}}t(a)^{ij}\subset U\mathfrak{t}_{n,N}
\]
be the degree $1$ part of $U\mathfrak{t}_{n,N}$. For $s\in\mathfrak{t}_{n,N}^{(1)}$,
define a $\mathbb{Q}$-linear map $\partial^{(s)}:A_{n,N}\to A_{n,N}$
by
\[
\partial^{(s)}\mathbb{I}(a_{0};a_{1},\dots,a_{k};a_{k+1})=\sum_{r\in\{\pm1\}}r\sum_{\substack{i=1\\
a_{i+r}\neq a_{i}
}
}^{k}\langle s,d\log(a_{i+r}-a_{i})\rangle\cdot\mathbb{I}(a_{0};a_{1},\dots,\widehat{a_{i}},\dots,a_{k};a_{k+1}).
\]
Then, for $\varphi\in U\mathfrak{t}_{n,N}$ and $u\in A_{n,N}$, the
pairing $\langle\varphi,\psi(u)\rangle$ can be calculated by the
following recursion formulas:

\begin{gather}
\langle1,\psi(\mathbb{I}(a_{0};a_{1},\dots,a_{k};a_{k+1}))=\delta_{k,0},\label{eq:pair_1_psi}\\
\langle s\cdot\varphi,\psi(w)\rangle=\langle\varphi,\psi(\partial^{(s)}w)\rangle.\label{eq:pair_s_psi}
\end{gather}

\section{\label{sec:Broadhurst-duality}Broadhurst duality}

Let $\mathcal{A}_{z}\coloneqq\mathbb{Q}\langle e_{0},e_{1},e_{z}\rangle$
be a free non-commutative algebra,
\[
\mathcal{A}_{z}^{0}\coloneqq\mathbb{Q}\oplus\mathbb{Q}e_{z}\oplus\bigoplus_{\substack{a\in\{1,z\}\\
b\in\{0,z\}
}
}e_{a}\mathcal{A}_{z}e_{b}\ \subset\mathcal{A}_{z}^{1}=\mathbb{Q}\oplus\bigoplus_{a\in\{1,z\}}e_{a}\mathcal{A}_{z}\ \subset\mathcal{A}_{z}
\]
the subspaces of $\mathcal{A}_{z}$, and $\tau_{z}$ the anti-automorphism
of $\mathcal{A}_{z}$ defined by $\tau_{z}(e_{0})=e_{z}-e_{1}$, $\tau_{z}(e_{1})=e_{z}-e_{0}$,
and $\tau_{z}(e_{z})=e_{z}$. For $c\in\{0,1\}$, define the $\mathbb{Q}$-linear
map $\partial_{z,c}:\mathcal{A}_{z}\to\mathcal{A}_{z}$ by
\[
\partial_{z,c}(e_{a_{1}}\cdots e_{a_{k}})\coloneqq\sum_{i=1}^{k}(\delta_{\{a_{i},a_{i+1}\},\{z,c\}}-\delta_{\{a_{i},a_{i-1}\},\{z,c\}})e_{a_{1}}\cdots\widehat{e_{a_{i}}}\cdots e_{a_{k}}
\]
with $e_{0}=0$ and $e_{a_{k+1}}=1$. In this section, we prove the
Broadhurst duality for the element of $\mathrm{PENT}(2,{\bf k})$
by using the following theorem.
\begin{thm}[{\cite[Theorem 10]{HS_AlgDif}, Algebraic differential for $\tau_{z}$}]
\label{thm:alg-diff-tauz}For $u\in\mathcal{A}_{z}^{0}$ and $c\in\{0,1\}$,
we have
\[
\tau_{z}^{-1}\circ\partial_{z,c}\circ\tau_{z}(u)=\partial_{z,c}(u).
\]
\end{thm}

We denote by $\iota$ the field embedding of $\mathbb{Q}(z,w)$ to
$F_{4,2}=\mathbb{Q}(z_{2},z_{3},z_{4})$ defined by
\[
\iota(z)=-\frac{z_{4}}{z_{3}},\iota(w)=\frac{z_{2}}{z_{3}}.
\]
Then 
\begin{align*}
d\log(\iota(0)-\iota(1))) & =0,\\
d\log(\iota(0)-\iota(z)) & =-d\log(z_{3})+d\log(z_{4}),\\
d\log(\iota(0)-\iota(w)) & =d\log(z_{2})-d\log(z_{3}),\\
d\log(\iota(1)-\iota(z)) & =d\log(z_{3}+z_{4})-d\log(z_{3}),\\
d\log(\iota(1)-\iota(w)) & =d\log(z_{2}-z_{3})-d\log(z_{3}),\\
d\log(\iota(z)-\iota(w)) & =d\log(z_{2}+z_{4})-d\log(z_{3}).
\end{align*}
Define a $\mathbb{Q}$-linear map $f:\mathcal{A}_{z}\to\mathcal{A}_{4,2}$
by $f(e_{a_{1}}\cdots e_{a_{k}})=\mathbb{I}(0;\iota(a_{1}),\dots,\iota(a_{k});\iota(w))$.
Then by Proposition \ref{prop:psi-in-Utn-dual}, $\psi(f(u))$ is
an element of the dual vector space of $U\mathfrak{t}_{4,2}$ for
$u\in\mathcal{A}_{z}$. Hereafter, we denote $t(0)^{ij}$ and $t(1)^{ij}$
in $U\mathfrak{t}_{4,2}^{0}$ by $t_{+}^{ij}$ and $t_{-}^{ij}$,
respectively. Recall that $\mathrm{PENT}(2,{\bf k})$ is defined as
the set of group-like power series $h\in\exp\mathfrak{t}_{3,2}^{0}$
satisfying
\[
h^{1,2,34}h^{12,3,4}=\delta(h)^{2,3,4}h^{1,23,4}h^{1,2,3}.
\]
Let us calculate $\left\langle h^{1,2,34}h^{12,3,4},\psi(f(u))\right\rangle $
and $\left\langle \delta(h)^{2,3,4}h^{1,23,4}h^{1,2,3},\psi(f(u))\right\rangle $
by using (\ref{eq:pair_1_psi}) and (\ref{eq:pair_s_psi}).
\begin{lem}
\label{lem:A1_is_stable}For $s\in\mathfrak{t}_{4,2}^{(1)}$, we have
\[
\partial^{(s)}\circ f(\mathcal{A}_{z}^{1})\subset f(\mathcal{A}_{z}^{1}).
\]
\end{lem}

\begin{proof}
It follows from the definition of $\partial^{(s)}$.
\end{proof}
\begin{lem}
\label{lem:pent_left}For $h=h(e^{0},e^{1},e^{-1})\in\exp\mathfrak{t}_{3,2}^{0}$
and $u\in\mathcal{A}_{z}^{1}$, we have
\[
\left\langle h^{1,2,34}h^{12,3,4},\psi(f(u))\right\rangle =\left\langle h,\left.u\right|_{z\to-1}\right\rangle .
\]
\end{lem}

\begin{proof}
First note that 
\begin{align*}
h^{1,2,34} & =h(t^{12},t_{+}^{23}+t_{+}^{24},t_{-}^{23}+t_{-}^{24}),\\
h^{12,3,4} & =h(t^{13}+t_{+}^{23}+t_{-}^{23},t_{+}^{34},t_{-}^{34}).
\end{align*}
By definition, we have
\begin{align*}
\partial^{(t_{12})}\circ f(e_{a_{1}}\cdots e_{a_{k}}) & =\delta_{a_{k},0}f(e_{a_{1}}\cdots e_{a_{k-1}}),\\
\partial^{(t_{+}^{23}+t_{+}^{24})}\circ f(e_{a_{1}}\cdots e_{a_{k}}) & =\delta_{a_{k},1}f(e_{a_{1}}\cdots e_{a_{k-1}}),\\
\partial^{(t_{-}^{23}+t_{-}^{24})}\circ f(e_{a_{1}}\cdots e_{a_{k}}) & =\delta_{a_{k},z}f(e_{a_{1}}\cdots e_{a_{k-1}}).
\end{align*}
Thus, by (\ref{eq:pair_1_psi}) and (\ref{eq:pair_s_psi}), we have
\[
\left\langle h^{1,2,34}h^{12,3,4},\psi(f(e_{a_{1}}\cdots e_{a_{k}}))\right\rangle =\sum_{j=0}^{k}\left\langle h,\left.(e_{a_{j+1}}\cdots e_{a_{k}})\right|_{z\to-1}\right\rangle \cdot\left\langle h^{12,3,4},\psi(f(e_{a_{1}}\cdots e_{a_{j}}))\right\rangle 
\]
where $u\mapsto u\mid_{z\to-1}$ is the ring homomorphism from $\mathcal{A}_{z}$
to $\mathfrak{h}_{2}=\mathbb{Q}\langle e_{0},e_{1},e_{-1}\rangle$
defined by $e_{z}\mapsto e_{-1}$ and $e_{a}\mapsto e_{a}$ for $a\in\{0,1\}$.
Since
\[
\partial^{(s)}\circ f(e_{a})=0\qquad(s\in\{t^{13}+t_{+}^{23}+t_{-}^{23},t_{+}^{34},t_{-}^{34}\},\ a\in\{1,z\})
\]
and $f(e_{a_{1}}\cdots e_{a_{j}})\in\mathcal{A}_{z}^{1}$, Lemma \ref{lem:A1_is_stable}
implies that
\[
\left\langle h^{12,3,4},\psi(f(e_{a_{1}}\cdots e_{a_{j}}))\right\rangle =\left\langle 1,\psi(f(e_{a_{1}}\cdots e_{a_{j}}))\right\rangle =\delta_{j,0}.
\]
Thus the lemma is proved.
\end{proof}
Let $u\mapsto u\mid_{z\to\infty}$ be the ring homomorphism from $\mathcal{A}_{z}$
to $\mathfrak{h}_{2}$ defined by $e_{z}\mapsto0$ and $e_{a}\mapsto e_{a}$
for $a\in\{0,1\}$.
\begin{lem}
\label{lem:pent-right}For $h=h(e^{0},e^{1},e^{-1})\in\exp\mathfrak{t}_{3,2}^{0}$
and $u\in\mathcal{A}_{z}^{0}$, we have
\[
\left\langle \delta(h)^{2,3,4}h^{1,23,4}h^{1,2,3},\psi(f(u))\right\rangle =\sum_{l=0}^{\infty}\sum_{c_{1},\dots,c_{l}\in\{0,1\}}\left\langle h,e_{-c_{1}}\cdots e_{-c_{l}}\right\rangle \cdot\left\langle h,\left.\left(D_{c_{1}}\circ\cdots\circ D_{c_{l}}(u)\right)\right|_{z\to\infty}\right\rangle .
\]
where we put $D_{0}=-\partial_{z,0}-\partial_{z,1}$ and $D_{1}=\partial_{z,1}$.
\end{lem}

\begin{proof}
First note that
\begin{align*}
h^{1,2,3} & =h(t^{12},t_{+}^{23},t_{-}^{23}),\\
h^{1,23,4} & =h(t^{12}+t^{13}+t_{+}^{23}+t_{-}^{23},t_{+}^{24}+t_{+}^{34},t_{-}^{24}+t_{-}^{34}),\\
\delta(h)^{2,3,4} & =h(t_{+}^{23},t_{+}^{34},0).
\end{align*}
Since
\begin{align*}
\partial^{(t_{+}^{23})}\circ f(e_{a_{1}}\cdots e_{a_{k}}) & =\delta_{a_{k},1}f(e_{a_{1}}\cdots e_{a_{k-1}}),\\
\partial^{(t_{+}^{34})}\circ f(e_{a_{1}}\cdots e_{a_{k}}) & =0,
\end{align*}
we have
\[
\left\langle \delta(h)^{2,3,4}h^{1,23,4}h^{1,2,3},\psi(f(u))\right\rangle =\left\langle h^{1,23,4}h^{1,2,3},\psi(f(u))\right\rangle .
\]
Furthermore, since
\begin{align*}
\partial^{(t^{12}+t^{13}+t_{+}^{23}+t_{-}^{23})}\circ f(u) & =f(D_{0}(u)),\\
\partial^{(t_{+}^{24}+t_{+}^{34})}\circ f(u) & =0,\\
\partial^{(t_{-}^{24}+t_{-}^{34})}\circ f(u) & =f(D_{1}(u)),
\end{align*}
we have
\[
\left\langle h^{1,23,4}h^{1,2,3},\psi(f(u))\right\rangle =\sum_{l=0}^{\infty}\sum_{c_{1},\dots,c_{l}\in\{0,1\}}\left\langle h,e_{-c_{1}}\cdots e_{-c_{l}}\right\rangle \cdot\left\langle h^{1,2,3},\psi(f(D_{c_{1}}\circ\cdots\circ D_{c_{l}}(u)))\right\rangle .
\]
Finally, since
\begin{align*}
\partial^{(t^{12})}\circ f(e_{a_{1}}\cdots e_{a_{k}}) & =\delta_{a_{k},0}f(e_{a_{1}}\cdots e_{a_{k-1}}),\\
\partial^{(t_{+}^{23})}\circ f(e_{a_{1}}\cdots e_{a_{k}}) & =\delta_{a_{k},1}f(e_{a_{1}}\cdots e_{a_{k-1}}),\\
\partial^{(t_{-}^{23})}\circ f(e_{a_{1}}\cdots e_{a_{k}}) & =0,
\end{align*}
we have
\[
\left\langle h^{1,2,3},\psi(f(w))\right\rangle =\left\langle h,\left.w\right|_{z\to\infty}\right\rangle 
\]
for any $w\in\mathbb{Q}\langle e_{0},e_{1},e_{z}\rangle$. Hence,
we have
\[
\left\langle \delta(h)^{2,3,4}h^{1,23,4}h^{1,2,3},\psi(f(u))\right\rangle =\sum_{l=0}^{\infty}\sum_{c_{1},\dots,c_{l}\in\{0,1\}}\left\langle h,e_{-c_{1}}\cdots e_{-c_{l}}\right\rangle \cdot\left\langle h,\left.\left(D_{c_{1}}\circ\cdots\circ D_{c_{l}}(u)\right)\right|_{z\to\infty}\right\rangle .\qedhere
\]
\end{proof}
Now we can show the Broadhurst duality for elements of ${\rm PENT}(2,{\bf k})$.
\begin{prop}[Broadhurst duality for mixed associator]
\label{prop:BroadhurstDuality}If $h\in\mathrm{PENT}(2,{\bf k})$
then
\[
h=\exp(-\alpha_{h}e^{1})\mathcal{T}(h)^{-1}\exp(-\alpha_{h}e^{0}),
\]
where $\alpha_{h}$ is the coefficient of $e^{-1}$ in $h$ and $\mathcal{T}$
is the involution of $U\mathfrak{t}_{3,2}^{0}$ defined by $e^{0}\leftrightarrow e^{1}$
and $e^{-1}\leftrightarrow e^{\infty}=-e^{0}-e^{1}-e^{-1}$.
\end{prop}

\begin{proof}
Since the both sides are group-like, it is enough to check
\[
\langle h,u\rangle=\langle\exp(-\alpha_{h}e^{1})\mathcal{T}(h)^{-1}\exp(-\alpha_{h}e^{0}),u\rangle
\]
for $u\in\{e_{0},e_{1}\}\cup\mathcal{A}^{0}(\{0,1,-1\})$. Since $h$
satisfies the mixed pentagon equation, the coefficients of $e^{0}$
and $e^{1}$ in $h$ are zero. Furthermore, by definition, we have
\[
\langle\exp(-\alpha_{h}e^{1})\mathcal{T}(h)^{-1}\exp(-\alpha_{h}e^{0}),e_{0}\rangle=\langle\exp(-\alpha_{h}e^{1})\mathcal{T}(h)^{-1}\exp(-\alpha_{h}e^{0}),e_{1}\rangle=0.
\]
Thus the case $u\in\{e_{0},e_{1}\}$ is proved. By Lemmas \ref{lem:pent_left}
and \ref{lem:pent-right}, we have
\[
\left\langle h,\left.u\right|_{z\to-1}\right\rangle =\sum_{l=0}^{\infty}\sum_{c_{1},\dots,c_{l}\in\{0,1\}}\left\langle h,e_{-c_{1}}\cdots e_{-c_{l}}\right\rangle \cdot\left\langle h,\left.\left(D_{c_{1}}\circ\cdots\circ D_{c_{l}}(u)\right)\right|_{z\to\infty}\right\rangle 
\]
for $u\in\mathcal{A}_{z}^{0}$. By putting $u=w-\tau_{z}(w)$ for
$w\in\mathcal{A}_{z}^{0}$, we have
\begin{align*}
 & \left\langle h,\left.w\right|_{z\to-1}\right\rangle -\left\langle h,\left.\tau_{z}(w)\right|_{z\to-1}\right\rangle \\
 & =\sum_{l=0}^{\infty}\sum_{c_{1},\dots,c_{l}\in\{0,1\}}\left\langle h,e_{-c_{1}}\cdots e_{-c_{l}}\right\rangle \cdot\left\langle h,\left.\left(D_{c_{1}}\circ\cdots\circ D_{c_{l}}(w)\right)\right|_{z\to\infty}\right\rangle \\
 & \quad-\sum_{l=0}^{\infty}\sum_{c_{1},\dots,c_{l}\in\{0,1\}}\left\langle h,e_{-c_{1}}\cdots e_{-c_{l}}\right\rangle \cdot\left\langle h,\left.\left(D_{c_{1}}\circ\cdots\circ D_{c_{l}}\circ\tau_{z}(w)\right)\right|_{z\to\infty}\right\rangle .
\end{align*}
Here, the last term can be calculated as
\begin{align*}
 & \left\langle h,\left.\left(D_{c_{1}}\circ\cdots\circ D_{c_{l}}\circ\tau_{z}(w)\right)\right|_{z\to\infty}\right\rangle \\
 & =\left\langle h,\left.\left(\tau_{z}\circ D_{c_{1}}\circ\cdots\circ D_{c_{l}}(w)\right)\right|_{z\to\infty}\right\rangle \qquad(\text{Theorem \ref{thm:alg-diff-tauz}})\\
 & =\left\langle h^{-1}(e_{1},e_{0},0),\left.\left(D_{c_{1}}\circ\cdots\circ D_{c_{l}}(w)\right)\right|_{z\to\infty}\right\rangle \quad(\text{\ensuremath{h} is group-like})\\
 & =\left\langle h(e_{0},e_{1},0),\left.\left(D_{c_{1}}\circ\cdots\circ D_{c_{l}}(w)\right)\right|_{z\to\infty}\right\rangle \qquad(\text{Duality for \ensuremath{\delta(h)} (Proposition \ref{prop:pent_property})})\\
 & =\left\langle h,\left.\left(D_{c_{1}}\circ\cdots\circ D_{c_{l}}(w)\right)\right|_{z\to\infty}\right\rangle .
\end{align*}
Thus, we have
\[
\left\langle h,\left.w\right|_{z\to-1}\right\rangle -\left\langle h,\left.\tau_{z}(w)\right|_{z\to-1}\right\rangle =0\qquad(w\in\mathcal{A}_{z}^{0}),
\]
which is equivalent to
\[
\langle h,u\rangle=\langle\mathcal{T}(h)^{-1},u\rangle\qquad(u\in\mathcal{A}^{0}(\{0,1,-1\})).
\]
Therefore the proposition is proved.
\end{proof}

\section{\label{sec:proof-of-main}The proof of main theorems}

In this section, we give a proof of main theorem. Let $h\in{\rm PENT}(2,{\bf k})$.
Then, we have
\begin{align}
 & \exp(-\alpha(t^{13}+t_{+}^{23}+t_{-}^{23}))(h^{1,2,3})^{-1}(h{}^{1,23,4})^{-1}\exp(-\alpha(t_{+}^{23}+t_{+}^{24}+t_{+}^{34}))\nonumber \\
 & =\exp(-\alpha(t^{13}+t_{+}^{23}+t_{-}^{23}))(h^{12,3,4})^{-1}(h^{1,2,34})^{-1}\delta(h)^{2,3,4}\exp(-\alpha(t_{+}^{23}+t_{+}^{24}+t_{+}^{34})).\label{eq:pen_e0}
\end{align}
We denote by $\kappa$ the field embedding of $\mathbb{Q}(z,w)$ to
$F_{4,2}=\mathbb{Q}(z_{2},z_{3},z_{4})$ defined by
\[
\kappa(z)=\frac{z_{4}-z_{3}}{z_{4}+z_{3}},\ \kappa(w)=\frac{z_{3}-z_{2}}{z_{3}+z_{2}}.
\]
Define a $\mathbb{Q}$-linear map $g:\mathbb{Q}\left\langle e_{0},e_{-1},e_{z},e_{-z^{2}}\right\rangle \to\mathcal{A}_{4,2}$
by
\[
g(e_{a_{1}}\cdots e_{a_{k}})=\mathbb{I}(0;\kappa(a_{1}/z),\dots,\kappa(a_{k}/z);\kappa(w)).
\]
By (\ref{eq:pen_e0}), we have
\begin{align}
 & \left\langle \exp(-\alpha(t^{13}+t_{+}^{23}+t_{-}^{23}))(h^{1,2,3})^{-1}(h{}^{1,23,4})^{-1}\exp(-\alpha(t_{+}^{23}+t_{+}^{24}+t_{+}^{34})),\psi(g(u))\right\rangle \nonumber \\
 & =\left\langle \exp(-\alpha(t^{13}+t_{+}^{23}+t_{-}^{23}))(h^{12,3,4})^{-1}(h^{1,2,34})^{-1}\delta(h)^{2,3,4}\exp(-\alpha(t_{+}^{23}+t_{+}^{24}+t_{+}^{34})),\psi(g(u))\right\rangle .\label{eq:pen_e0_pair}
\end{align}
In this section, we prove the theorem by computing the both sides
of the above equation for $u\in\mathcal{B}$. Note that the components
of $\mathbb{I}(0;\kappa(a_{1}/z),\dots,\kappa(a_{k}/z);\kappa(w))$
are $0$, $1$, $-z$, $-z^{-1}$, or $w$. We remark that $d\log(\kappa(a)-\kappa(s'))$
for $a,a'\in\{0,1,-z,-z^{-1},w\}$ are given by
\begin{align*}
d\log(\kappa(0)-\kappa(1)) & =0,\\
{\rm d}\log(\kappa(0)-\kappa(-z)) & =d\log(z_{3}-z_{4})-d\log(z_{3}+z_{4}),\\
d\log(\kappa(0)-\kappa(-z^{-1})) & =-d\log(z_{3}-z_{4})+d\log(z_{3}+z_{4}),\\
d\log(\kappa(0)-\kappa(w)) & =d\log(z_{2}-z_{3})-d\log(z_{2}+z_{3}),\\
d\log(\kappa(1)-\kappa(-z)) & =d\log(z_{4})-d\log(z_{3}+z_{4}),\\
d\log(\kappa(1)-\kappa(-z^{-1})) & =d\log(z_{4})-d\log(z_{3}-z_{4}),\\
d\log(\kappa(1)-\kappa(w)) & =d\log(z_{2})-d\log(z_{2}+z_{3}),\\
d\log(\kappa(-z)-\kappa(-z^{-1})) & =d\log(z_{3})+d\log(z_{4})-d\log(z_{3}-z_{4})-d\log(z_{3}+z_{4}),\\
d\log(\kappa(-z)-\kappa(w)) & =d\log(z_{3})+d\log(z_{2}-z_{4})-d\log(z_{2}+z_{3})-d\log(z_{3}+z_{4}),\\
d\log(\kappa(-z^{-1})-\kappa(w)) & =d\log(z_{3})+d\log(z_{2}+z_{4})-d\log(z_{2}+z_{3})-d\log(z_{3}-z_{4}).
\end{align*}

\begin{lem}
\label{lem:B'_is_stable}Put $\mathcal{B}'=\mathbb{Q}\oplus\bigoplus_{a\in\{-1,z,-z^{2}\}}e_{a}\mathbb{Q}\langle e_{0},e_{-1},e_{z},e_{-z^{2}}\rangle$.
Then for $s\in\mathfrak{t}_{4,2}^{(1)}$, we have
\[
\partial^{(s)}\circ g(\mathcal{B}')\subset g(\mathcal{B}').
\]
\end{lem}

\begin{proof}
It follows from the definition of $\partial^{(s)}$.
\end{proof}
\begin{lem}
\label{lem:left_step1}For $h\in{\rm PENT}(2,{\bf k})$ and $u\in\mathcal{B}$,
we have

\[
\left\langle (h^{1,2,3})^{-1}(h{}^{1,23,4})^{-1}\exp(-\alpha(t_{+}^{23}+t_{+}^{24}+t_{+}^{34})),\psi(g(u))\right\rangle =\left\langle h,u\mid_{z\to1}\right\rangle .
\]
\end{lem}

\begin{proof}
First note that
\[
h^{1,23,4}=h(t^{12}+t^{13}+t_{+}^{23}+t_{-}^{23},t_{+}^{24}+t_{+}^{34},t_{-}^{24}+t_{-}^{34}).
\]
Furthermore, $\partial^{(s)}(g(e_{a}))=0$ for any $s\in\{t^{12}+t^{13}+t_{+}^{23}+t_{-}^{23},t_{+}^{24}+t_{+}^{34},t_{-}^{24}+t_{-}^{34},t_{+}^{23}+t_{+}^{24}+t_{+}^{34}\}$
and $a\in\{z,-1,-z^{2}\}$. Thus, by Lemma \ref{lem:B'_is_stable},
we have
\[
\left\langle (h^{1,2,3})^{-1}(h{}^{1,23,4})^{-1}\exp(-\alpha(t_{+}^{23}+t_{+}^{24}+t_{+}^{34})),g(u)\right\rangle =\left\langle (h^{1,2,3})^{-1},g(u)\right\rangle 
\]
for $u\in\mathcal{B}$. By Broadhurst duality (Proposition \ref{prop:BroadhurstDuality}),
we have
\begin{equation}
(h^{1,2,3})^{-1}=\exp(\alpha t^{12})h(t_{+}^{23},t^{12},-t^{12}-t_{+}^{23}-t_{-}^{23})\exp(\alpha t_{+}^{23}).\label{eq:h123^-1}
\end{equation}
Then we have
\begin{align*}
\partial^{(t_{+}^{23})}\circ g(e_{a_{1}}\cdots e_{a_{k}}) & =\delta_{a_{k},0}g(e_{a_{1}}\cdots e_{a_{k-1}}),\\
\partial^{(t_{12})}\circ g(e_{a_{1}}\cdots e_{a_{k}}) & =\delta_{a_{k},z}g(e_{a_{1}}\cdots e_{a_{k-1}}),\\
\partial^{(-t^{12}-t_{+}^{23}-t_{-}^{23})}\circ g(e_{a_{1}}\cdots e_{a_{k}}) & =(\delta_{a_{k},-1}+\delta_{a_{k},-z^{2}})g(e_{a_{1}}\cdots e_{a_{k-1}})
\end{align*}
for $a_{1},\dots,a_{k}\in\{0,1,-z,-z^{2}\}$. Thus, for $e_{a_{1}}\cdots e_{a_{k}}\in\mathcal{B}$,
we have
\begin{align*}
 & \left\langle (h^{1,2,3})^{-1},\psi(g(e_{a_{1}}\cdots e_{a_{k}}))\right\rangle \\
 & =\left\langle \exp(\alpha t^{12})h(t_{+}^{23},t^{12},-t^{12}-t_{+}^{23}-t_{-}^{23})\exp(\alpha t_{+}^{23}),\psi(g(e_{a_{1}}\cdots e_{a_{k}}))\right\rangle \qquad(\text{by (\ref{eq:h123^-1})})\\
 & =\left\langle h(t_{+}^{23},t^{12},-t^{12}-t_{+}^{23}-t_{-}^{23})\exp(\alpha t_{+}^{23}),\psi(g(e_{a_{1}}\cdots e_{a_{k}}))\right\rangle \qquad(\text{by \ensuremath{a_{k}\neq z}})\\
 & =\sum_{j=0}^{k}\left\langle h,\left.(e_{a_{j+1}}\cdots e_{a_{k}})\right|_{z\to1}\right\rangle \cdot\left\langle \exp(\alpha t_{+}^{23}),\psi(g(e_{a_{1}}\cdots e_{a_{j}}))\right\rangle \\
 & =\sum_{j=0}^{k}\left\langle h,\left.(e_{a_{j+1}}\cdots e_{a_{k}})\right|_{z\to1}\right\rangle \cdot\delta_{j,0}\qquad(\text{by \ensuremath{a_{1}\neq0}})\\
 & =\left\langle h,\left.(e_{a_{1}}\cdots e_{a_{k}})\right|_{z\to1}\right\rangle .
\end{align*}
Hence the lemma is proved.
\end{proof}
\begin{lem}
\label{lem:left_step2}For $h\in{\rm PENT}(2,{\bf k})$ and $u\in\mathcal{B}$,
we have

\[
\left\langle \exp(-\alpha(t^{13}+t_{+}^{23}+t_{-}^{23}))(h^{1,2,3})^{-1}(h{}^{1,23,4})^{-1}\exp(-\alpha(t_{+}^{23}+t_{+}^{24}+t_{+}^{34})),\psi(g(u))\right\rangle =\left\langle h,u\mid_{z\to1}\right\rangle .
\]
\end{lem}

\begin{proof}
By definition, for $e_{a_{1}}\cdots e_{a_{k}}\in\mathcal{B}$, we
have
\begin{align*}
\partial^{(t^{13}+t_{+}^{23}+t_{-}^{23})}(g(e_{a_{1}}\cdots e_{a_{n}})) & =g(\partial_{1}(e_{a_{1}}\cdots e_{a_{n}}))
\end{align*}
where 
\[
\partial_{1}:\mathcal{B}\to\mathcal{B}\quad;\quad e_{a_{1}}\cdots e_{a_{n}}\mapsto\sum_{i=1}^{n-1}\delta_{\{a_{i},a_{i+1}\},\{-1,-z^{2}\}}e_{a_{1}}\cdots e_{a_{i-1}}(e_{a_{i+1}}-e_{a_{i}})e_{a_{i+2}}\cdots e_{a_{n}}
\]
 is the operator defined in Section \ref{sec:confluence}. By definition
of $\partial_{1}$, we have
\[
\partial_{1}(\mathcal{B})\subset\ker(\mathcal{B}\xrightarrow{u\mapsto u\mid_{z\to1}}\mathfrak{h}_{2})
\]
and therefore
\begin{equation}
\left.\left((\partial_{1})^{n}(u)\right)\right|_{z\to1}=0\label{eq:vanish1}
\end{equation}
for $u\in\mathcal{B}$ and $n>0$. Hence, for $u\in\mathcal{B}$,
we have
\begin{align*}
 & \left\langle \exp(-\alpha(t^{13}+t_{+}^{23}+t_{-}^{23}))(h^{1,2,3})^{-1}(h{}^{1,23,4})^{-1}\exp(-\alpha(t_{+}^{23}+t_{+}^{24}+t_{+}^{34})),\psi(g(u))\right\rangle \\
 & =\sum_{n=0}^{\infty}\frac{(-\alpha)^{n}}{n!}\langle(h^{1,2,3})^{-1}(h{}^{1,23,4})^{-1}\exp(-\alpha(t_{+}^{23}+t_{+}^{24}+t_{+}^{34})),\psi\circ g\circ(\partial_{1})^{n}(u)\rangle\\
 & =\sum_{n=0}^{\infty}\frac{(-\alpha)^{n}}{n!}\left\langle h,\left.\left((\partial_{1})^{n}(u)\right)\right|_{z\to1}\right\rangle \qquad(\text{Lemma \ref{lem:left_step2}})\\
 & =\left\langle h,u\mid_{z\to1}\right\rangle \qquad(\text{by (\ref{eq:vanish1})}),
\end{align*}
which completes the proof.
\end{proof}
\begin{lem}
\label{lem:pair_h234}Let $\lambda:\mathbb{Q}\left\langle e_{0},e_{-1},e_{z},e_{-z^{2}}\right\rangle \to\mathfrak{h}_{2}$
be a ring homomorphism defined by
\[
\lambda(e_{-1})=\lambda(e_{z})=0,\ \lambda(e_{0})=e_{0}-e_{1},\ \lambda(e_{-z^{2}})=-e_{1}.
\]
Then, for $u\in\mathbb{Q}\langle e_{0},e_{-1},e_{-z},e_{-z^{2}}\rangle$,
we have
\[
\left\langle \delta(h)^{2,3,4},\psi(g(u))\right\rangle =\left\langle h,\lambda(u)\right\rangle .
\]
\end{lem}

\begin{proof}
First, note that 
\begin{align*}
\delta(h)^{2,3,4} & =h(t_{+}^{23},t_{+}^{34},0).
\end{align*}
Since
\[
\left\langle t_{+}^{23},d\log((-z^{-1})-a)\right\rangle =0,\quad\left\langle t_{+}^{34},d\log((-z^{-1})-a)\right\rangle =-1
\]
for $a\in\{0,1,-z,w\}$, we have
\[
\left\langle \delta(h)^{2,3,4},f(e_{a_{1}}\cdots e_{a_{k}})\right\rangle =0.
\]
if $a_{j}=-z^{-1}$ for some $j$. Furthermore, since
\[
\left\langle t_{+}^{23},d\log(1-a)\right\rangle =\left\langle t_{+}^{34},d\log(1-a)\right\rangle =0
\]
for $a\in\{0,-z,w\}$, we have
\[
\left\langle \delta(h)^{2,3,4},g(e_{a_{1}}\cdots e_{a_{k}})\right\rangle =0
\]
if $a_{j}\in\{-z^{-1},1\}$ for some $j$. Since
\begin{align*}
\left\langle t_{+}^{23},d\log(0-(-z))\right\rangle  & =0\\
\left\langle t_{+}^{23},d\log(0-w)\right\rangle  & =1\\
\left\langle t_{+}^{23},d\log((-z)-w)\right\rangle  & =0\\
\left\langle t_{+}^{34},d\log(0-(-z))\right\rangle -1 & =0\\
\left\langle t_{+}^{34},d\log(0-w)\right\rangle -1 & =-1\\
\left\langle t_{+}^{34},d\log((-z)-w)\right\rangle -1 & =-1,
\end{align*}
we have
\[
\left\langle \delta(h)^{2,3,4},g(u)\right\rangle =\left\langle h,\lambda(u)\right\rangle ,
\]
which completes the proof.
\end{proof}
\begin{lem}
\label{lem:pair_exp(alpha_t12)}For $u\in\mathcal{B}''$, we have
\[
\left\langle \exp(\alpha t^{12})\mathcal{T}(h^{1,2,34})\delta(h)^{2,3,4},\psi(g(u))\right\rangle =\left\langle h,{\rm reg}_{z\to0}(u)\right\rangle .
\]
\end{lem}

\begin{proof}
Note that we have
\[
\mathcal{T}(h^{1,2,34})=h(t_{+}^{23}+t_{+}^{24},t^{12},-t^{12}-t_{+}^{23}-t_{+}^{24}-t_{-}^{23}-t_{-}^{24})
\]
and
\begin{align*}
\partial^{(t_{+}^{23}+t_{+}^{24})}(g(e_{a_{1}}\cdots e_{a_{k}})) & =(\delta_{a_{k},0}+\delta_{a_{k},-z^{2}})g(e_{a_{1}}\cdots e_{a_{k-1}}),\\
\partial^{(t^{12})}(g(e_{a_{1}}\cdots e_{a_{k}})) & =\delta_{a_{k},z}g(e_{a_{1}}\cdots e_{a_{k-1}}),\\
\partial^{(-t^{12}-t_{+}^{23}-t_{+}^{24}-t_{-}^{23}-t_{-}^{24})}(g(e_{a_{1}}\cdots e_{a_{k}})) & =0.
\end{align*}
Thus, for $e_{a_{1}}\cdots e_{a_{k}}\in\mathcal{B}$, we have 
\begin{align*}
 & \left\langle \exp(\alpha t^{12})\mathcal{T}(h^{1,2,34})\delta(h)^{2,3,4},\psi(g(e_{a_{1}}\cdots e_{a_{k}}))\right\rangle \\
 & =\left\langle \mathcal{T}(h^{1,2,34})\delta(h)^{2,3,4},\psi(g(e_{a_{1}}\cdots e_{a_{k}}))\right\rangle \qquad(a_{k}\neq z)\\
 & =\sum_{j=0}^{k}\left\langle \delta(h)^{2,3,4},\psi\circ g(e_{a_{1}}\cdots e_{a_{j}})\right\rangle \cdot\left\langle h,\overline{{\rm reg}}_{z\to0}(e_{a_{j+1}}\cdots e_{a_{k}})\right\rangle \\
 & =\sum_{j=0}^{k}\left\langle h,\lambda(e_{a_{1}}\cdots e_{a_{j}})\right\rangle \cdot\left\langle h,\overline{{\rm reg}}_{z\to0}(e_{a_{j+1}}\cdots e_{a_{k}})\right\rangle \qquad(\text{by Lemma (\ref{lem:pair_h234})}).
\end{align*}
If $a_{i}=z$ for some $1\leq i\leq k$, then
\begin{equation}
\sum_{j=0}^{k}\left\langle h,\lambda(e_{a_{1}}\cdots e_{a_{j}})\right\rangle \cdot\left\langle h,\overline{{\rm reg}}_{z\to0}(e_{a_{j+1}}\cdots e_{a_{k}})\right\rangle =0.\label{eq:zero_if_aj_z}
\end{equation}
If $e_{a_{1}}\cdots e_{a_{k}}\in\mathcal{B}''$ then we have
\begin{align}
 & \sum_{j=0}^{k}\left\langle h,\lambda(e_{a_{1}}\cdots e_{a_{j}})\right\rangle \cdot\left\langle h,\overline{{\rm reg}}_{z\to0}(e_{a_{j+1}}\cdots e_{a_{k}})\right\rangle \nonumber \\
 & =\sum_{j=0}^{k}\delta_{j,0}\left\langle h,\overline{{\rm reg}}_{z\to0}(e_{a_{j+1}}\cdots e_{a_{k}})\right\rangle \cdot\qquad(\text{by \ensuremath{a_{1}=z}})\nonumber \\
 & =\left\langle h,\overline{{\rm reg}}_{z\to0}(e_{a_{1}}\cdots e_{a_{k}})\right\rangle \nonumber \\
 & =\left\langle h,{\rm dist}\circ\overline{{\rm reg}}_{z\to0}(e_{a_{1}}\cdots e_{a_{k}})\right\rangle \qquad(\text{Proposition \ref{prop:pent_property}})\label{eq:B''_case}
\end{align}
If $e_{a_{1}}\cdots e_{a_{k}}\in\mathcal{B}'''$ then
\begin{align}
 & \sum_{j=0}^{k}\left\langle h,\lambda(e_{a_{1}}\cdots e_{a_{j}})\right\rangle \cdot\left\langle h,\overline{{\rm reg}}_{z\to0}(e_{a_{j+1}}\cdots e_{a_{k}})\right\rangle \cdot\nonumber \\
 & =\sum_{j=0}^{k}\left\langle h,\lambda(e_{a_{1}}\cdots e_{a_{j}})\right\rangle \cdot\left\langle h,e_{0}^{k-j}\right\rangle \qquad(a_{j+1},\dots,a_{k}\in\{0,-z^{2}\})\nonumber \\
 & =\sum_{j=0}^{k}\left\langle h,\lambda(e_{a_{1}}\cdots e_{a_{j}})\right\rangle \cdot\delta_{k-j}\qquad(\text{the coefficient of \ensuremath{e^{0}} in \ensuremath{h} is \ensuremath{0}})\nonumber \\
 & =\langle h,\lambda(e_{a_{1}}\cdots e_{a_{k}}).\nonumber \\
 & =\left\langle h,\varrho(\epsilon(e_{a_{1}}\cdots e_{a_{k}}))\right\rangle \nonumber \\
 & =\left\langle h,\wp(\epsilon(e_{a_{1}}\cdots e_{a_{k}}))\right\rangle \qquad(\text{Lemma (\ref{eq:phico_eq_wp})})\label{eq:B'''_case}
\end{align}
Thus, for $u\in\mathcal{B}''$ and $v\in\mathcal{B}''$, we have
\begin{align}
 & \left\langle \exp(\alpha t^{12})\mathcal{T}(h^{1,2,34})\delta(h)^{2,3,4},\psi(g(u\shuffle v))\right\rangle \nonumber \\
 & =\left\langle \exp(\alpha t^{12})\mathcal{T}(h^{1,2,34})\delta(h)^{2,3,4},\psi(g(u))\right\rangle \cdot\left\langle \exp(\alpha t^{12})\mathcal{T}(h^{1,2,34})\delta(h)^{2,3,4},\psi(g(v))\right\rangle \nonumber \\
 & =\left\langle h,{\rm dist}(\overline{{\rm reg}}_{z\to0}(u))\right\rangle \cdot\left\langle h,\wp(\epsilon(v))\right\rangle \qquad(\text{by (\ref{eq:B''_case}) and (\ref{eq:B'''_case})})\nonumber \\
 & =\left\langle h,{\rm dist}(\overline{{\rm reg}}_{z\to0}(u))\shuffle\wp(\epsilon(v))\right\rangle \nonumber \\
 & =\left\langle h,{\rm reg}_{z\to0}(u\shuffle v)\right\rangle .\label{eq:B''_B'''_case}
\end{align}
Now, the lemma follows from (\ref{eq:zero_if_aj_z}) and (\ref{eq:B''_B'''_case}).
\end{proof}
\begin{lem}
\label{lem:right_part_final}For $u\in\mathcal{B}$, we have
\[
\left\langle \exp(-\alpha(t^{13}+t_{+}^{23}+t_{-}^{23}))(h^{12,3,4})^{-1}(h^{1,2,34})^{-1}\delta(h)^{2,3,4}\exp(-\alpha(t_{+}^{23}+t_{+}^{24}+t_{+}^{34})),\psi(g(u))\right\rangle =\langle h,\varphi(u)\rangle.
\]
\end{lem}

\begin{proof}
We have 
\begin{align*}
 & \exp(-\alpha(t^{13}+t_{+}^{23}+t_{-}^{23}))(h^{12,3,4})^{-1}(h^{1,2,34})^{-1}\delta(h)^{2,3,4}\exp(-\alpha(t_{+}^{23}+t_{+}^{24}+t_{+}^{34}))\\
 & =\mathcal{T}(h^{12,3,4})\exp(\alpha t_{+}^{34})(h^{1,2,34})^{-1}\delta(h)^{2,3,4}\exp(-\alpha(t_{+}^{23}+t_{+}^{24}+t_{+}^{34}))\qquad(\text{Broadhurst duality})\\
 & =\mathcal{T}(h^{12,3,4})(h^{1,2,34})^{-1}\exp(-\alpha(t_{+}^{23}+t_{+}^{24}))\delta(h)^{2,3,4}\qquad([t_{+}^{34},(h^{1,2,34})^{-1}]=[\delta(h)^{2,3,4},t_{+}^{23}+t_{+}^{24}+t_{+}^{34}])\\
 & =\mathcal{T}(h^{12,3,4})\exp(\alpha t^{12})\mathcal{T}(h^{1,2,34})\delta(h)^{2,3,4}\qquad(\text{Broadhurst duality}),
\end{align*}
and thus
\begin{align}
 & \left\langle \exp(-\alpha(t^{13}+t_{+}^{23}+t_{-}^{23}))(h^{12,3,4})^{-1}(h^{1,2,34})^{-1}\delta(h)^{2,3,4}\exp(-\alpha(t_{+}^{23}+t_{+}^{24}+t_{+}^{34})),\psi(g(u))\right\rangle \nonumber \\
 & =\left\langle \mathcal{T}(h^{12,3,4})\exp(\alpha t^{12})\mathcal{T}(h^{1,2,34})\delta(h)^{2,3,4},\psi(g(u))\right\rangle .\label{eq:proof_main_step1}
\end{align}
We have
\[
\mathcal{T}(h^{12,3,4})=h(t_{+}^{34},t^{13}+t_{+}^{23}+t_{-}^{23},-t^{13}-t_{+}^{23}-t_{-}^{23}-t_{+}^{34}-t_{-}^{34})
\]
and
\begin{align*}
\partial^{(t_{+}^{34})}g(u) & =g(\partial_{0}(u)),\\
\partial^{(t^{13}+t_{+}^{23}+t_{-}^{23})}g(u) & =g(\partial_{1}(u)),\\
\partial^{(-t^{13}-t_{+}^{23}-t_{-}^{23}-t_{+}^{34}-t_{-}^{34})}g(u) & =g(\partial_{-1}(u)).
\end{align*}
Thus we have
\begin{align}
 & \left\langle \mathcal{T}(h^{12,3,4})\exp(\alpha t^{12})\mathcal{T}(h^{1,2,34})\delta(h)^{2,3,4},\psi(g(u))\right\rangle \nonumber \\
 & =\sum_{l=0}^{\infty}\sum_{c_{1},\dots,c_{l}\in\{0,1,-1\}}\left\langle h,e_{c_{1}}\cdots e_{c_{l}}\right\rangle \cdot\left\langle \exp(\alpha t^{12})\mathcal{T}(h^{1,2,34})\delta(h)^{2,3,4},g(\partial_{c_{1}}\cdots\partial_{c_{l}}u)\right\rangle \nonumber \\
 & =\sum_{l=0}^{\infty}\sum_{c_{1},\dots,c_{l}\in\{0,1,-1\}}\left\langle h,e_{c_{1}}\cdots e_{c_{l}}\right\rangle \cdot\left\langle h,{\rm reg}_{z\to0}(\partial_{c_{1}}\cdots\partial_{c_{l}}(u))\right\rangle \qquad(\text{Lemma \ref{lem:pair_exp(alpha_t12)}})\nonumber \\
 & =\sum_{l=0}^{\infty}\sum_{c_{1},\dots,c_{l}\in\{0,1,-1\}}\left\langle h,{\rm reg}_{\shuffle}(e_{c_{1}}\cdots e_{c_{l}})\right\rangle \cdot\left\langle h,{\rm reg}_{z\to0}(\partial_{c_{1}}\cdots\partial_{c_{l}}u)\right\rangle \nonumber \\
 & =\left\langle h,{\rm reg}_{z\to0}(\partial_{c_{1}}\cdots\partial_{c_{l}}u)\shuffle{\rm reg}_{\shuffle}(e_{c_{1}}\cdots e_{c_{l}})\right\rangle \nonumber \\
 & =\left\langle h,\varphi(u)\right\rangle .\label{eq:proof_main_step2}
\end{align}
Hence the lemma follows from (\ref{eq:proof_main_step1}) and (\ref{eq:proof_main_step2}).
\end{proof}
Now, we can prove the main theorem of this paper.
\begin{proof}[Proof of Theorem \ref{thm:Gpent-ICF}]
Recall that $\mathcal{I}_{{\rm CF}}$ is defined by
\[
\mathcal{I}_{{\rm CF}}\coloneqq\{u\mid_{z\to1}-\varphi(u)\mid u\in\mathcal{B}\}.
\]
For $h\in{\rm PENT}(2,{\bf k})$ and $u\in\mathcal{B}$, by (\ref{eq:pen_e0_pair})
and Lemmas \ref{lem:left_step2} and \ref{lem:right_part_final},
we have
\[
\langle h,u\mid_{z\to1}-\varphi(u)\rangle=0.
\]
Thus Theorem \ref{thm:Gpent-ICF} is proved. Furthermore, by Theorem
\ref{thm:kerLm=00003DICF}, we have 
\[
\langle h,u\rangle=0\qquad(h\in{\rm PENT}(2,{\bf k}),\,u\in\ker L^{\mathfrak{m}}),
\]
and thus
\[
\langle h,u\rangle=0\qquad(h\in\varphi\in{\rm GRTM}_{(\bar{1},1)}(2,{\bf k}),\,u\in\ker L^{\mathfrak{a}}).
\]
As explained in the introduction, this implies Theorem \ref{thm:main}.
\end{proof}

\subsection*{Acknowledgements}

The author would like to thank Hidekazu Furusho and Kenji Sakugawa
for useful comments and advices. This work was supported by JSPS KAKENHI
Grant Numbers JP18K13392 and JP22K03244.


\begin{thebibliography}{10}
\bibitem{Bro_mix}F. Brown, `Mixed Tate motives over $\mathbb{Z}$',
Ann. Math., 175 (2012), 949-976.

\bibitem{Deli_es}P. Deligne, `Le groupe fondamental unipotent motivique
de $\mathbb{G}_{m}-\mu_{N}$, pour $N=2,3,4,6$ ou $8$', Publ. Math.
Inst. Hautes Etudes Sci. (2010), 101--141.

\bibitem{DelGon}P. Deligne and A. B. Goncharov, `Groupes fondamentaux
motiviques de Tate mixte', Annales scientifiques de l'École Normale
Supérieure 38.1 (2005): 1-56.

\bibitem{Dr_quasi}V. G. Drinfeld, `On quasitriangular quasi-Hopf
algebras and a group closely connected with Gal($\bar{\mathbb{Q}}/\mathbb{Q}$)',
Leningrad Math. J. 2 (1991), no. 4, 829--860.

\bibitem{Enr_quasi}B. Enriquez, `Quasi-reflection algebras and cyclotomic
associators', Selecta Math. (N.S.) 13 (2007), 391--463.

\bibitem{EnrFur}B. Enriquez and H. Furusho, `Mixed pentagon, octagon
and Broadhurst duality equations', J. Pure Appl. Algebra 216 (2012),
982-995.

\bibitem{Fur_hex}H. Furusho, `Pentagon and hexagon equations', Ann.
of Math. 171 (2010), 545-556.

\bibitem{Fur_dsh}H. Furusho, `Double shuffle relation for associators',
Ann. of Math. 174 (2011), 341-360.

\bibitem{Gro}A. Grothendieck, `Esquisse d\textquoteright un programme,'
in \textquotedblleft Geometric Galois actions, 1\textquotedblright ,
London Math. Soc. Lecture Note Ser., 242, (1997) 5--48.

\bibitem{HS_AlgDif}M. Hirose and N. Sato, `Algebraic differential
formulas for the shuffle, stuffle and duality relations of iterated
integrals,' J. Algebra 556 (2020), 363-384.

\bibitem{HS_EulerSum}M. Hirose and N. Sato, `The motivic Galois group
of mixed Tate motives over $\mathbb{Z}[1/2]$ and its action on the
fundamental group of $\mathbb{P}^{1}\setminus\{0,\pm1,\infty\}$',
arXiv:2007.04288v2 {[}math.NT{]}.

\bibitem{Ihara}Y. Ihara, Some arithmetic aspects of Galois actions
in the pro-$p$ fundamental group of $\mathbb{P}^{1}-\{0,1,\infty\}$.
In: Arithmetic fundamental groups and noncommutative algebra (Berkeley,
CA, 1999), Proc. Sympos. Pure Math., 70, Amer. Math. Soc., Providence,
RI, 2002, 247--273
\end{thebibliography}
\end{document}